\documentclass[12pt]{amsart}
\usepackage{amscd,amssymb,amsthm,amsmath,csquotes,enumerate}
\usepackage[matrix,arrow]{xy}

\newcommand{\mumu}{\boldsymbol{\mu}}

\theoremstyle{definition}
\newtheorem{example}[equation]{Example}
\newtheorem*{example*}{Example}

\newtheorem{theorem}[equation]{Theorem}
\newtheorem{lemma}[equation]{Lemma}
\newtheorem{corollary}[equation]{Corollary}
\newtheorem{proposition}[equation]{Proposition}

\newtheorem*{conjecture*}{Conjecture}

\newtheorem*{question*}{Question}
\newtheorem{problem}[equation]{Problem}
\newtheorem*{problem*}{Problem}
\newtheorem*{theorem*}{Theorem}

\theoremstyle{remark}
\newtheorem{remark}[equation]{Remark}
\newtheorem*{remark*}{Remark}

\makeatletter\@addtoreset{equation}{section} \makeatother

\author{Ivan Cheltsov and Constantin Shramov}

\title{K\"ahler--Einstein Fano threefolds of degree $22$}

\address{\emph{Ivan Cheltsov}
\newline
\textnormal{School of Mathematics, The University of Edinburgh,  Edinburgh EH9 3JZ, UK.}
\newline
\textnormal{\texttt{I.Cheltsov@ed.ac.uk}}}

\address{\emph{Constantin Shramov}
\newline
\textnormal{Steklov Mathematical Institute of RAS,
8 Gubkina street, Moscow 119991, Russia.}
\newline
\textnormal{National Research University Higher School of Economics, Laboratory of Algebraic Geometry, NRU HSE, 6 Usacheva str., Moscow, 117312, Russia.}
\newline
\textnormal{\texttt{costya.shramov@gmail.com}}}

\begin{document}

\begin{abstract}
We study the problem of existence of K\"ahler--Einstein metrics on smooth Fano threefolds of Picard rank one and anticanonical degree $22$ that admit a faithful action of the multiplicative group $\mathbb{C}^\ast$.
We prove that, with a possible exception of two explicitly described cases, all such smooth Fano threefolds are K\"ahler--Einstein.
\end{abstract}

\maketitle

All varieties are assumed to be projective and defined over the field of complex numbers.

\section{Introduction}
\label{section:intro}

Smooth Fano threefolds of Picard rank $1$ have been classified by Iskovskikh in \cite{Iskovskikh1977,Iskovskikh1978}.
Among them, he found a family missing in the original works by Fano.
Threefolds in this family have the same cohomology groups as $\mathbb{P}^3$ does.
Their anticanonical degree is~$22$, hence they are called threefolds of type~$V_{22}$.
In fact, Iskovskikh himself missed one threefold in this family, which was later recovered by Mukai and Umemura in~\cite{MukaiUmemura}.
This threefold, usually called the Mukai--Umemura threefold,
is an equivariant compactification of $\mathrm{SL}_2(\mathbb{C})\slash\mathbf{I}$, where $\mathbf{I}$ denotes the icosahedral group.
Its automorphism group is isomorphic to the group $\mathrm{PGL}_2(\mathbb{C})$.

The automorphism groups of threefolds of type $V_{22}$ have been studied by Prokhorov in \cite{Prokhorov}.
He proved that this group is finite except for a unique threefold
for which the connected component of identity of the automorphism group is isomorphic to  the additive group $\mathbb{C}^{+}$;
and a one-parameter family of threefolds that admit a faithful action of the multiplicative group $\mathbb{C}^\ast$,
which includes the Mukai--Umemura threefold as a special member.
We refer to the latter varieties as threefolds of type $V_{22}^\ast$.

In \cite{Tian1997}, Tian showed that there are threefolds of type $V_{22}$ with trivial automorphism group
that do not admit K\"ahler--Einstein metrics, which disproved a folklore conjecture that all smooth Fano varieties
without holomorphic vector fields are K\"ahler--Einstein.
On the other hand, Donaldson proved

\begin{theorem}[{\cite[Theorem~3]{Donaldson}}]
\label{theorem:Donaldson}
Let $X$ be the Mukai--Umemura threefold, and $G$ be its automorphism group. Then
$$
\alpha_G\big(X\big)=\frac{5}{6}.
$$
\end{theorem}

Here $\alpha_G(X)$ is the $G$-equivariant $\alpha$-invariant defined by Tian in \cite{Tian1987}.
If $X$ is a smooth Fano variety, and $G$ is a reductive subgroup in $\mathrm{Aut}(X)$,
then Demailly's \cite[Theorem~A.3]{CheltsovShramovUMN} gives
\begin{equation}
\label{equation:Demailly}
\alpha_G(X)=\mathrm{sup}\left\{\epsilon\in\mathbb{Q}\ \left|\ \aligned
&\text{the log pair}\ \left(X, \frac{\epsilon}{n}\mathcal{D}\right)\ \text{is log canonical for any}\ n\in\mathbb{Z}_{>0}\\
&\text{and every $G$-invariant linear system}\ \mathcal{D}\subset\big|-nK_{X}\big|\\
\endaligned\right.\right\}.
\end{equation}

Donaldson's Theorem~\ref{theorem:Donaldson} implies the existence of a K\"ahler--Einstein metric on the Mukai--Umemura threefold by famous Tian's criterion:

\begin{theorem}[{\cite{Tian1987}}]
\label{theorem:Tian}
Let $X$ be a smooth Fano variety of dimension $n$, and $G$ be a reductive subgroup in $\mathrm{Aut}(X)$.
Suppose that
$$
\alpha_G\big(X\big)>\frac{n}{n+1}.
$$
Then $X$ admits a K\"ahler--Einstein metric.
\end{theorem}

An example of a K\"ahler--Einstein threefold of type $V_{22}$ with finite automorphism group has been constructed in \cite{CheltsovShramovKlein}.
On the other hand, there exist threefolds of this type that are not K\"ahler--Einstein.

\begin{example}\label{example:Xu}
Let $X^{\mathrm{a}}$ be the unique threefold of type $V_{22}$
such that the connected component of identity of its automorphism group is isomorphic to  the additive group $\mathbb{C}^{+}$.
By the Matsushima obstruction, the variety $X^{\mathrm{a}}$ is not K\"ahler--Einstein.
It is interesting to point out that $X^{\mathrm{a}}$ is K-semistable. Indeed, it follows from \cite[Proposition~5.4.4]{KuznetsovProkhorovShramov} and the Mukai construction of varieties of type $V_{22}$ (cf. \cite[Remark~5.4.8]{KuznetsovProkhorovShramov})  that
the Mukai--Umemura threefold is a degeneration of $X^{\mathrm{a}}$.
Since the Mukai--Umemura threefold is K\"ahler--Einstein, it is K-polystable by \cite{ChenDonaldsonSun}, so that in particular it is K-semistable.
On the other hand, K-semistability is an open condition, see \cite[Theorem~1.4]{Xu} or~\mbox{\cite[Corollary~1.2]{BlumLiuXu}}. Hence~$X^{\mathrm{a}}$ is K-semistable.
\end{example}

The problem of existence of K\"ahler--Einstein metrics on threefolds of type $V_{22}^\ast$
was addressed by Donaldson in \cite{Donaldson,Donaldson2017}, by Rollin, Simanca and Tipler in \cite{RollinSimancaTipler},
and by Dinew, Kapustka and Kapustka in \cite{DinewKapustkaKapustka}.
In particular, they proved that the set of such threefolds that are K\"ahler--Einstein is open in moduli in the Euclidean topology.
Donaldson suggested that in fact all threefolds of type $V_{22}^\ast$ are K\"ahler--Einstein.
In \cite{Donaldson}, he wrote
\begin{displayquote}
The Mukai--Umemura manifold has $\tau=1$.
When $\tau$ is close to $1$ we have seen that the corresponding manifold admits a K\"ahler--Einstein metric.
It seems likely that this true for all $\tau$ but, as far as the author is aware, this is not known.
It seems an interesting test case for future developments in the existence theory.
\end{displayquote}
Here $\tau$ is a parameter in the moduli space of threefolds of type $V_{22}^\ast$ that is used in \cite{Donaldson}.
The Mukai--Umemura threefold corresponds to $\tau=1$.

In \cite[\S4.1]{Donaldson2017}, Donaldson made a more precise suggestion about which threefolds of type $V_{22}$ are K\"ahler--Einstein metric and which are not.
It also predicts that each threefold of type $V_{22}^\ast$ must admit a K\"ahler--Einstein metric.

To verify Donaldson's suggestion,
Dinew, Kapustka and Kapustka estimated the \mbox{$\alpha_{\mathbb{C}^\ast}$-invariants} of threefolds of type $V_{22}^\ast$.
It appeared that they do not exceed $\frac{1}{2}$,
so that Tian's Theorem~\ref{theorem:Tian} cannot be applied.
However, the automorphism groups of all threefolds of type~$V_{22}^\ast$ are actually larger than $\mathbb{C}^\ast$.
It was pointed out in \cite{RollinSimancaTipler,DinewKapustkaKapustka} that there exists an additional involution
that anti-commutes with the $\mathbb{C}^\ast$-action, so that together they generate a subgroup isomorphic to $\mathbb{C}^\ast\rtimes\mumu_2$.
Here $\mumu_2$ denotes the group of order $2$.
In fact, by \cite[Theorem~1.3]{KuznetsovProkhorov}, one has
$$
\mathrm{Aut}\big(X\big)\cong\mathbb{C}^\ast\rtimes\mumu_2
$$
for every threefold $X$ of type $V_{22}^\ast$ that is not the Mukai--Umemura threefold.

Dinew, Kapustka and Kapustka posed

\begin{problem}[{\cite[Problem~7.1]{DinewKapustkaKapustka}}]
\label{problem:DinewKapustkaKapustka}
Let $X$ be a smooth Fano threefold of type $V_{22}^{\ast}$, and let $G$ be a subgroup in $\mathrm{Aut}(X)$
that is isomorphic to $\mathbb{C}^\ast\rtimes\mumu_2$.
Compute $\alpha_G(X)$.
\end{problem}

In this paper we completely solve this problem using the description of smooth Fano threefolds of type  $V_{22}^\ast$
obtained recently by Kuznetsov and Prokhorov in \cite{KuznetsovProkhorov}.

Kuznetsov and Prokhorov proved that the isomorphisms classes of Fano threefolds of type $V_{22}^\ast$ are naturally parameterized by $u\in\mathbb{C}\setminus\{0,1\}$.
In~\S\ref{section:explicit}, we present their construction in details.
Note that the parameter $u$ used by Kuznetsov and Prokhorov in \cite{KuznetsovProkhorov} differs from the parameter $\tau$ used by Donaldson in \cite{Donaldson}.

To state our main result, we denote by $V_u$ the smooth Fano threefold of type  $V_{22}^\ast$
that corresponds to the parameter $u$ in the construction of \cite{KuznetsovProkhorov}.
Then the Mukai--Umemura threefold is $V_u$ for $u=-\frac{1}{4}$ by \cite[Theorem~1.3]{KuznetsovProkhorov}.
Let $G$ a subgroup in $\mathrm{Aut}(V_u)$ such that
$$
G\cong\mathbb{C}^\ast\rtimes\mumu_2.
$$
The main result of our paper is

\begin{theorem}
\label{theorem:main}
One has
$$
\alpha_G(V_u)=
\left\{\aligned
&\frac{4}{5}\ \ \text{if}\ u\ne\frac{3}{4}\ \text{and}\ u\ne 2,\\
&\frac{3}{4}\ \ \text{if}\ u=\frac{3}{4},\\
&\frac{2}{3}\ \ \text{if}\ u=2.\\
\endaligned
\right.
$$
\end{theorem}

Applying Tian's Theorem~\ref{theorem:Tian}, we obtain

\begin{corollary}
\label{corollary:main}
If $u\ne\frac{3}{4}$ and $u\ne 2$, then $V_u$ is K\"ahler--Einstein.
\end{corollary}

\begin{remark}
\label{remark:Kento}
If $u=\frac{3}{4}$ or $u=2$, then $V_u$ is also K\"ahler--Einstein.
This has been recently proved by Fujita in \cite{Fujita2021}.
Note also that Theorem~\ref{theorem:main} and \cite[Theorem~1.4.10]{ACCFKMGSSV}
imply that $V_u$ is K\"ahler--Einstein for $u=\frac{3}{4}$.
\end{remark}

Let us describe the scheme of the proof of Theorem~\ref{theorem:main}.
To estimate $\alpha_G(V_u)$, one has to describe irreducible $G$-invariant subvarieties of small degree in $V_u$.
Since $G$ acts on $V_u$ without fixed points, we have to deal with irreducible $G$-invariant curves of small degree,
and $G$-invariant anticanonical surfaces in $V_u$. However, the geometry of the~threefold $V_u$ is rather complicated,
and it is hard to complete these tasks in a straightforward way. Instead, we use a construction of the threefold $V_u$
as a $G$-equivariant birational image of a smooth quadric hypersurface in $\mathbb{P}^4$ found
recently by Kuznetsov and Prokhorov in \cite{KuznetsovProkhorov}, see the diagram \eqref{equation:V22} for more details.
This allows to describe irreducible $G$-invariant curves of small degree in $V_u$
and $G$-invariant surfaces in $|-K_{V_u}|$ in terms of the quadric, whose $G$-equivariant geometry is much easier to control.
In particular, this description gives us an upper bound on $\alpha_G(V_u)$.
To show that the latter bound is sharp, we have to study $G$-equivariant birational geometry of the threefold $V_u$.
We do this using three explicit $G$-equivariant Sarkisov links that start from $V_u$.
As a result, we obtain the formula for $\alpha_G(V_u)$ in Theorem~\ref{theorem:main}.

Let us describe the structure of this paper.
In \S\ref{section:explicit}, we recall from \cite{KuznetsovProkhorov} the explicit
construction of the threefold $V_u$ using a birational map from a three-dimensional quadric.
In this section, we also describe this birational map explicitly in coordinates.
In \S\ref{section:curves}, we start an explicit classification of
irreducible $G$-invariant curves of small degree in the threefold $V_u$,
which will be used in the proof of Theorem~\ref{theorem:main}.
In \S\ref{section:curves-2}, we complete this classification, see Proposition~\ref{proposition:curves}.
In \S\ref{section:pencil}, we study the pencil in the linear system $|-K_{V_u}|$ that consists of all $G$-invariant surfaces
and describe singularities of surfaces in this pencil.
In \S\ref{section:link}, we describe one Sarkisov link that plays a crucial role in the proof of Theorem~\ref{theorem:main}.
In this section, we also describe two special birational transformations of the threefold~$V_u$,
which are known as \emph{bad Sarkisov links}. They are also used in the proof of our Theorem~\ref{theorem:main}.
Finally, in \S\ref{section:proof}, we prove Theorem~\ref{theorem:main}.

\bigskip

\textbf{Acknowledgements.}
The authors are very grateful to Hamid Abban, Sir Simon Donaldson, Kento Fujita, Alexander Kuznetsov,
Yuri Prokhorov, Cristiano Spotti, and Chenyang Xu for useful discussions.
The work of Constantin Shramov was performed at the Steklov
International Mathematical Center and supported by the Ministry of
Science and Higher Education of the Russian Federation (agreement no.~\mbox{075-15-2019-1614}).
He was also supported by the Russian Academic Excellence Project~\mbox{``5-100''}
and the Young Russian Mathematics award.

\section{Kuznetsov--Prokhorov construction}
\label{section:explicit}

Consider the projective space $\mathbb{P}^4$ with homogeneous coordinates $x$, $y$, $z$, $t$, and $w$.
Suppose that the group $\mathbb{C}^\ast$ act on $\mathbb{P}^4$ by
\begin{equation}
\label{equation:action}
\lambda\colon (x:y:z:t:w)\mapsto (x:\lambda y:\lambda^3 z:\lambda^5 t:\lambda^6 w).
\end{equation}
Furthermore, consider the involution $\iota$ acting on $\mathbb{P}^4$ by
\begin{equation}\label{equation:involution}
\iota\colon (x:y:z:t:w)\mapsto (w:t:z:y:x).
\end{equation}
This defines the action of the group $G\cong\mathbb{C}^\ast\rtimes\mumu_2$ on $\mathbb{P}^4$.

Let the quadric $Q_u$, where $u\in\mathbb{C}$, be given by equation
\begin{equation}\label{equation:quadric}
u(xw-z^2)+(z^2-yt)=0.
\end{equation}
Then the quadric $Q_u$ is $G$-invariant.
Note that $Q_u$ is smooth provided that $u\not\in\{0,1\}$.
Therefore, until the end of the paper (with the only exception of Remark~\ref{remark:V1} below),
we will always assume that neither $u=0$ nor $u=1$.

Let $\Gamma$ be the image of $\mathbb{P}^1$ with homogeneous coordinates $(s_0:s_1)$
embedded into $\mathbb{P}^4$ by
$$
(s_0:s_1)\mapsto (s_0^6:s_0^5s_1:s_0^3s_1^3:s_0s_1^5:s_1^6).
$$
Then $\Gamma$ is a $G$-invariant curve contained in the quadric $Q_u$.
It is the closure of the $G$-orbit of the point $(1:1:1:1:1)$.
One easily checks that $\mathrm{deg}(\Gamma)=6$,
cf. Lemma~\ref{lemma:parameterization} below.

Let $\mathcal{S}$ be the complete intersection in $\mathbb{P}^4$ that is given by
$$
\left\{\aligned
&xw-z^2=0,\\
&z^2-yt=0.
\endaligned
\right.
$$
Then the surface $\mathcal{S}$ is $G$-invariant,
and $\Gamma\subset\mathcal{S}\subset Q_u$.

\begin{remark}
\label{remark:dP4}
The surface $\mathcal{S}$ is a toric singular del Pezzo surface of degree $4$ that has $4$ ordinary double points.
These points are $(1:0:0:0:0)$, $(0:0:0:0:1)$, $(0:1:0:0:0)$ and $(0:0:0:1:0)$.
The first two of them  are contained in the curve $\Gamma$.
\end{remark}

It was proved in \cite[Theorem~2.2]{KuznetsovProkhorov} (cf.~\cite[(2.13.2)]{Takeuchi})
that there exists the following $G$-equivariant commutative diagram
\begin{equation}
\label{equation:V22}
\xymatrix{
&\widetilde{Q}_u\ar@{-->}[rr]^{\chi}\ar@{->}[ldd]_{\pi}\ar@{->}[rd]^{\alpha}&& \widetilde{V}_u\ar@{->}[ld]_{\beta}\ar@{->}[rdd]^{\phi}&\\%
&&Y_{u}&&\\
Q_u\ar@{-->}[rrrr]^{\zeta}\ar@{-->}[rru]^{\gamma}&&&&V_u\ar@{-->}[llu]_{\omega}}
\end{equation} %
Here $V_u$ is a smooth Fano threefold of type $V_{22}^\ast$,
the morphism $\pi$ is the blow up of the quadric $Q_u$ along the curve $\Gamma$,
the morphism $\phi$ is the blow up of the threefold $V_u$ along a (unique) $G$-invariant smooth rational curve $\mathcal{C}_2$ with $-K_{V_u}\cdot \mathcal{C}_2=2$,
the map $\chi$ is a flop in two smooth rational curves, which we will describe later in Remark~\ref{remark:flop}.
The morphisms $\alpha$ and $\beta$ in \eqref{equation:V22} are small birational morphisms
that are given by the linear systems $|-nK_{\widetilde{Q}_u}|$
and $|-nK_{\widetilde{V}_u}|$ for $n\gg 0$, respectively.
By construction, the threefold~$Y_{u}$ is a non-$\mathbb{Q}$-factorial Fano threefold with terminal singularities such that $-K_{Y_{u}}^3=16$.

\begin{remark}
\label{remark:V22-classification}
Kuznetsov and Prokhorov showed in \cite{KuznetsovProkhorov} that every smooth Fano threefold of type $V_{22}^\ast$
can be obtained via diagram~\eqref{equation:V22} for some $u\in\mathbb{C}\setminus\{0,1\}$.
Moreover, they proved that for distinct $u$ the resulting varieties $V_u$ are not isomorphic.
Furthermore, if~\mbox{$u=-\frac{1}{4}$}, then $V_u$ is the Mukai--Umemura threefold by \cite[Theorem~1.3]{KuznetsovProkhorov}.
For other descriptions of threefolds of type $V_{22}^\ast$, see \cite[\S5.3]{Donaldson},
\mbox{\cite[\S2.2]{DinewKapustkaKapustka}} and \cite[\S5.3]{KuznetsovProkhorovShramov}.
\end{remark}

Recall from \cite[Proposition~4.1.11]{IskovskikhProkhorov} that the divisor $-K_{V_u}$ is very ample,
and the linear system $|-K_{V_u}|$ gives an embedding $V_u\hookrightarrow\mathbb{P}^{13}$.
In particular, the curve $\mathcal{C}_2$ is a conic in this embedding.
Let us identify  $V_u$ with its anticalonical image in $\mathbb{P}^{13}$ and fix the following notation.
\begin{itemize}
\item We denote by $H_{Q_{u}}$ a hyperplane section of the quadric $Q_u$ in $\mathbb{P}^4$.
\item We denote by $H_{V_{u}}$ a hyperplane section of the threefold $V_u$ in $\mathbb{P}^{13}$.
\item We denote by $\widetilde{\mathcal{S}}$ the proper transform of the surface $\mathcal{S}$ on the threefold $\widetilde{Q}_u$.
\item We denote by $E_{Q_{u}}$ the exceptional surface of the blow up $\pi$.
\item We denote by $E_{V_{u}}$ the exceptional surface of the blow up $\phi$.
\end{itemize}
Then $\widetilde{\mathcal{S}}$ is the proper transform of $E_{V_{u}}$ on $\widetilde{Q}_u$,
which is the unique divisor in the linear system $|2\pi^*(H_{Q_{u}})-E_{Q_{u}}|$.
Similarly, the proper transform of $E_{Q_{u}}$ on $\widetilde{V}_u$ is the unique surface in the linear system $|2\phi^*(H_{V_u})-5E_{V_{u}}|$.
Thus, we also fix the following notation.
\begin{itemize}
\item We denote by $\widetilde{\mathcal{R}}$ the unique surface in the linear system $|2\phi^*(H_{V_u})-5E_{V_{u}}|$.
\item We denote by $\mathcal{R}$ the proper transform of the surface $\widetilde{\mathcal{R}}$ on the threefold $V_u$.
\end{itemize}

\begin{corollary}
\label{corollary:4-5}
One has $\alpha_G(V_u)\leqslant\frac{4}{5}$.
\end{corollary}

\begin{proof}
Let $D=\frac{1}{2}\mathcal{R}$. Then $D\sim_{\mathbb{Q}}-K_{V_u}$.
Moreover, since $\mathcal{R}\sim -2K_{V_{u}}$ and $\mathrm{mult}_{\mathcal{C}_2}(\mathcal{R})=5$,
the log pair $(V_u,\frac{4}{5}D)$ is not Kawamata log terminal. Indeed, we have
$$
K_{\widetilde{V}_u}+\frac{4}{5}\widetilde{D}+E_{V_{u}}\sim_{\mathbb{Q}}\phi^*\Big(K_{V_u}+\frac{4}{5}D\Big).
$$
This shows that $\alpha_G(V_u)\leqslant\frac{4}{5}$.
\end{proof}

Using the information about the classes of the exceptional divisors $E_{Q_{u}}$ and $E_{V_{u}}$, one can easily check
that the rational map
$\phi\circ\chi\colon\widetilde{Q}_u\dasharrow V_u$ is given by the linear system~\mbox{$|5\pi^*(H_{Q_{u}})-2E_{Q_{u}}|$},
and the rational map~\mbox{$\pi\circ\chi^{-1}\colon\widetilde{V}_u\dasharrow Q_u$} is given by the linear
system~\mbox{$|\phi^*(H_{V_u})-2E_{V_{u}}|$}.

\begin{remark}
\label{remark:Y16}
By \cite[Proposition~4.1.12(iii)]{IskovskikhProkhorov}, the threefold $V_u$ is a scheme-theoretic intersection of quadrics in $\mathbb{P}^{13}$.
Thus, since $-K_{\widetilde{V}_u}\sim \phi^*(H_{V_u})-E_{V_{u}}$ and~\mbox{$h^0(\mathcal{O}_{\widetilde{V}_u}(-K_{\widetilde{V}_u}))=11$},
the linear system $|-K_{\widetilde{V}_u}|$ gives a morphism  $V_{u}\to\mathbb{P}^{10}$ that is birational on its image.
Hence, there is a commutative diagram
$$
\xymatrix{
&\widetilde{V}_u\ar@{->}[ld]\ar@{->}[rd]^{\phi}&\\%
\mathbb{P}^{10}&&V_u\ar@{-->}[ll]}
$$
such that the dashed arrow is a linear projection from the conic $\mathcal{C}_2$.
This implies that we can assume that the morphism $\beta$ in~\eqref{equation:V22}
is given by the linear system $|-K_{\widetilde{V}_u}|$.
Hence, we can also assume that the morphism $\alpha$ is given by the linear system $|-K_{\widetilde{Q}_{u}}|$.
Thus, the threefold $Y_u$ is a (singular) Fano threefold anticanonically embedded into $\mathbb{P}^{10}$.
\end{remark}

Let $L_{1}$ and $L_{2}$ be the tangent lines in $\mathbb{P}^4$ to the curve $\Gamma$ at the points  $(1:0:0:0:0)$ and $(0:0:0:0:1)$, respectively.
Then $L_{1}$ is given by
\begin{equation}
\label{equation:L1}
z=t=w=0,
\end{equation}
and the line $L_{2}$ is given by
\begin{equation}
\label{equation:L2}
x=y=z=0.
\end{equation}
Thus, both lines $L_{1}$ and $L_{2}$ are contained in the surface $\mathcal{S}$.
Denote by $\widetilde{L}_{1}$ and $\widetilde{L}_{2}$ the proper transforms of the lines $L_{1}$ and $L_{2}$ on the threefold $\widetilde{Q}_u$, respectively.

\begin{remark}
\label{remark:flop}
By \cite[Remark~5.3]{KuznetsovProkhorov}, the curves $\widetilde{L}_{1}$ and $\widetilde{L}_{2}$ are the flopping curves of the map $\chi$.
The flopping curves of $\chi^{-1}$ are described in \cite[Proposition~5.1]{KuznetsovProkhorov}.
Namely, the threefold $V_u$ contains exactly two lines that intersect the conic $\mathcal{C}_2$.
Denote them~by $\ell_1$ and~$\ell_2$, and denote their proper transforms on $\widetilde{V}_u$ by $\widetilde{\ell}_1$ and $\widetilde{\ell}_2$, respectively.
The lines~$\ell_1$ and~$\ell_2$ intersect the conic $\mathcal{C}_2$ transversally, because $V_u$ is an intersection of quadrics.
Moreover, the lines $\ell_1$ and $\ell_2$ are contained in the surface $\mathcal{R}$,
since $\mathcal{R}\sim -2K_{V_{u}}$
and~\mbox{$\mathrm{mult}_{\mathcal{C}_2}(\mathcal{R})=5$}.
By~\mbox{\cite[Remark~5.3]{KuznetsovProkhorov}}, the curves $\widetilde{\ell}_1$ and $\widetilde{\ell}_2$ are exactly the flopping curves of the map~$\chi^{-1}$.
Thus, the birational map $\zeta$ in \eqref{equation:V22} induces an isomorphism
$$
Q_v\setminus\mathcal{S}\cong V_u\setminus\mathcal{R}.
$$
Without loss of generality, we may assume that
$\beta(\widetilde{\ell}_1)=\alpha(\widetilde{L}_{1})$ and $\beta(\widetilde{\ell}_2)=\alpha(\widetilde{L}_{2})$.
Note that the lines $\ell_1$ and $\ell_2$ on the Fano threefold $V_u$ are \emph{special}, i.e., their normal bundles in $V_u$
are isomorphic to $\mathcal{O}_{\mathbb{P}^1}(1)\oplus\mathcal{O}_{\mathbb{P}^1}(-2)$; see the proof of~~\mbox{\cite[Proposition~5.1]{KuznetsovProkhorov}}.
This implies that the normal bundles of the curves $\widetilde{\ell}_1$ and $\widetilde{\ell}_2$ in $\widetilde{V}_u$
are isomorphic to $\mathcal{O}_{\mathbb{P}^1}\oplus\mathcal{O}_{\mathbb{P}^1}(-2)$,
so that the flop $\chi^{-1}$ is given by Reid's pagoda \cite[\S5]{Reid}.
\end{remark}

\begin{remark}
\label{remark:V1}
It follows from Theorem~\ref{theorem:main} and Remark~\ref{remark:Kento} that $V_u$ is K-polystable for every $u\not\in\{0,1\}$.
It would be interesting to find the K-polystable limits of the threefolds $V_u$ when $u\to 0$, $u\to 1$ and $u\to \infty$.
In fact, we have a candidate for the limit in the case when $u\to 1$.
Namely, if $u=1$, then the quadric threefold $Q_u$ is singular at the point~\mbox{$(0:0:1:0:0)$}.
This point is not contained in the surface $\mathcal{S}$,
and it is not contained in the curve $\Gamma$.
Thus, the commutative diagram \eqref{equation:V22} still makes sense in this case.
The threefold $V_1$ is a Fano threefold with one ordinary double point such that~\mbox{$-K_{V_1}^3=22$}.
By \cite[Proposition~5.4]{KuznetsovProkhorov}, one has $\mathrm{Pic}(V_1)\cong\mathbb{Z}$ and $\mathrm{Cl}(V_1)\cong\mathbb{Z}^2$,
so that $V_1$ is one of the threefolds described in~\mbox{\cite[Theorem~1.2]{Prokhorov2016}}.
Note also that $\mathrm{Cl}(V_1)^G\cong\mathbb{Z}^2$.
We expect that $V_1$ is K-polystable, so that it is the K-polystable limit of our threefolds $V_u$ when $u\to 1$.
\end{remark}

The commutative diagram~\eqref{equation:V22} is a Sarkisov link (that starts at $Q_u$ and ends at $V_u$).
It plays a crucial role in the proof of our Theorem~\ref{theorem:main}.
In \S\ref{section:link}, we describe another $G$-equivariant Sarkisov link that starts at $V_u$ and ends at another
threefold of type $V_{22}^\ast$ (possibly isomorphic to~$V_u$).
This link also helps to prove Theorem~\ref{theorem:main}.

\begin{remark}[{cf. \cite{CheltsovShramovKlein,CheltsovShramovA6,CheltsovShramovQuartic,CheltsovShramovBOOK,CheltsovShramovP3}}]
\label{remark:G-link}
It would be interesting to study other $G$-Sarkisov links that start at the threefold $V_u$ or the quadric $Q_u$.
Such links usually arise from $G$-irreducible curves of small degree or $G$-orbits of small length.
For example, the inverse of the link~\eqref{equation:V22} arises from the conic $\mathcal{C}_2$, which is irreducible and $G$-invariant.
The curve $\ell_1+\ell_2$ from Remark~\ref{remark:flop} also gives rise to a $G$-Sarkisov link.
Namely, one can show that there exists a $G$-equivariant commutative diagram
\begin{equation}
\label{equation:G-link}
\xymatrix{
&\overline{V}_u\ar@{-->}[rr]^{\varrho}\ar@{->}[ld]_{\upsilon}\ar@{->}[rd]^{\varsigma}&&\overline{W}\ar@{->}[ld]_{\varphi}\ar@{->}[rd]^{\nu}&\\%
V_u&&U&&W}
\end{equation}
Here $\upsilon$ is a blow up of the lines $\ell_1$ and $\ell_2$,
the morphisms $\varsigma$ and $\varphi$ are small and birational,
the map $\varrho$ flops the curves contracted by $\varsigma$,
the threefold $U$ is a Fano threefold with terminal singularities such that $-K_U^3=14$,
the threefold $W$ is a smooth Fano threefold such that $\mathrm{Pic}(W)\cong\mathbb{Z}^2$ and $-K_W^3=28$,
and $\nu$ is a birational morphism that contracts the proper transform of the unique surface in
$|-K_{V_u}|$ which is singular along the lines $\ell_1$ and~$\ell_2$ to a smooth rational curve of (anticanonical) degree $6$.
Note that $\mathrm{Pic}(W)^G\cong\mathbb{Z}$, and $W$ is the threefold No.~(1.2.3) in \cite[Theorem~1.2]{Prokhorov2013}.
It can be realized as the blow-up of a smooth quadric in $\mathbb{P}^4$ along a twisted quartic curve.
Note that unlike~\eqref{equation:V22} the diagram \eqref{equation:G-link} is not a Sarkisov link in the usual sense \cite{Corti95},
because the curve~\mbox{$\ell_1+\ell_2$} is reducible.
\end{remark}

Now we describe the birational maps $\gamma$ and $\zeta$ in the  diagram \eqref{equation:V22} explicitly using coordinates on $\mathbb{P}^4$.
To describe the map $\gamma$, recall that this map is given by the restriction of the linear system of all cubic hypersurfaces in $\mathbb{P}^4$
that pass through the curve $\Gamma$ to the quadric $Q_u$.
Since $\gamma$ is $G$-equivariant and, in particular, $\mathbb{C}^\ast$-equivariant,
we are in position to choose $\mathbb{C}^\ast$-invariant generators of this linear system.
To start with, set
$$
f=xw-yt,
$$
so that the equation $f=0$ cuts out the surface $\mathcal{S}$ on the quadric $Q_u$. Then we set
\begin{multline}
\label{equation:H-3}
h_3=y^3-x^2z, \quad
h_5=x^2t-y^2z, \quad
h_6=xf, \quad
h_7=yf,\\
h_8=y^2w-xzt, \quad
h_9=zf, \quad
h_{10}=xt^2-yzw, \quad
h_{11}=tf,\\
h_{12}=wf, \quad
h_{13}=yw^2-zt^2, \quad
h_{15}=t^3-zw^2.
\end{multline}
Then the involution $\iota$ swaps the polynomials $h_{i}$ and $h_{18-i}$ for $3\leqslant i\leqslant 8$,
and it preserves the polynomial $h_9$. Observe also that these $11$ cubic polynomials all vanish on the curve~$\Gamma$.
Moreover, the corresponding surfaces in $Q_u$ cut out by $h_i=0$ are smooth at a general point of the curve $\Gamma$,
so that their proper transforms on $\widetilde{Q}_u$ are all contained in the linear system
$|-K_{\widetilde{Q}_{u}}|=|3\pi^*(H_{Q_{u}})-E_{Q_{u}}|$.

Every polynomial $h_i$ is semi-invariant  with respect to the $\mathbb{C}^\ast$-action~\eqref{equation:action}.
Moreover, the weight of the polynomial $h_i$ equals $i$.
This implies, in particular, that they define linearly independent sections in $H^0(\mathcal{O}_{Q_u}(3H_{Q_u}))$.
Since $h^0(\mathcal{O}_{\widetilde{Q}_{u}}(-K_{\widetilde{Q}_{u}}))=11$
by the Riemann--Roch formula and Kawamata--Viehweg vanishing theorem,
we conclude that the birational map $\gamma$ in \eqref{equation:V22} is given by
\begin{equation}
\label{equation:gamma}
(x:y:z:t:w)\mapsto\big(h_3:h_5:h_6:h_7:h_8:h_9:h_{10}:h_{11}:h_{12}:h_{13}:h_{15}\big).
\end{equation}
Thus, using \eqref{equation:L1} and \eqref{equation:L2},
we see that $\gamma(L_1)=(1:0:0:0:0:0:0:0:0:0:0:0)$ and $\gamma(L_2)=(0:0:0:0:0:0:0:0:0:0:0:1)$.

Now let us describe the map $\zeta$ in in \eqref{equation:V22}.
To do this, we set
\begin{equation}
\label{equation:3-5-6-7-8-9-10-11-12-13-15}
g_{i+6}=f\cdot h_i
\end{equation}
for $i\in\{3,5,6,7,8,9,10,11,12,13,15\}$. Let
\begin{multline}
\label{equation:10-15-20}
g_{10}=(u-1)x^2yzw-3xy^2zt+(2-u)xyz^3+y^4w+x^3t^2,\\
g_{20}=(u-1)xztw^2-3yzt^2w+(2-u)z^3tw+xt^4+y^2w^3,\\
g_{15}^\prime=(u-1)x^2t^3+(u-1)y^3w^2-(u+4)y^2zt^2+(3u+2)xyztw+(4-4u)yz^3t.
\end{multline}
Note that the involution $\iota$ swaps the polynomials $g_{i}$ and $g_{30-i}$ for $9\leqslant i\leqslant 14$,
and it preserves both polynomials $g_{15}$ and $g_{15}^\prime$.
Observe that all polynomials $g_i$ and the polynomial~$g_{15}^\prime$ are semi-invariant
with respect to the $\mathbb{C}^\ast$-action~\eqref{equation:action}.
Moreover, the weight of the polynomial $g_i$ equals $i$,
and the weight of the polynomial $g_{15}^\prime$ equals $15$.
Also observe that
$$
g_{15}^\prime(0,1,0,0,1)=1\ne 0=g_{15}(0,1,0,0,1),
$$
and the point $(0:1:0:0:1)$ is contained in the quadric $Q_u$.
This implies, in particular, that these $14$ quintic polynomials define linearly independent sections in $H^0(\mathcal{O}_{Q_u}(5H_{Q_u}))$.

For every $i\in\{9,\ldots,21\}$, denote by $M_i$ the surface in the quadric $Q_u$ that is cut out by the equation $g_i=0$.
Similarly, denote by $M_{15}^\prime$ the surface in $Q_u$ that is cut out by the equation $g_{15}^\prime=0$.
It is easy to see that all these surfaces pass through the curve $\Gamma$.

\begin{lemma}
\label{lemma:quintics}
The surfaces $M_i$ and $M_{15}^\prime$ are singular along $\Gamma$.
\end{lemma}

\begin{proof}
For $i\in\{3,5,6,7,8,9,10,11,12,13,15\}$ this follows from the fact that the polynomials $h_i$ and $f$ vanish along $\Gamma$.
To check the assertion for the surfaces $M_{10}$, $M_{20}$ and~$M_{15}^\prime$, one can just write down the partial derivatives
of $g_{10}$, $g_{20}$ and $g_{15}^\prime$ at the point~\mbox{$(1:1:1:1:1)$}, compare them with the partial derivatives of the left hand side of~\eqref{equation:quadric}, and then use the fact that $\Gamma$ is the closure of the orbit of the latter point.
\end{proof}

One can check that the multiplicities of the surfaces $M_i$ and $M_{15}^\prime$ along the curve $\Gamma$ equal~$2$.
This also follows from the fact that the surfaces $E_{Q_u}$ and $\widetilde{\mathcal{S}}$ generate the cone of effective divisors of the threefold $\widetilde{Q}_u$.
We conclude that the proper transforms of the surfaces $M_i$ and $M_{15}^\prime$ on the threefold $\widetilde{Q}_u$ generate the linear system $|5H_{Q_u}-2E_{Q_u}|$.
Hence, the birational map $\zeta$ in \eqref{equation:V22} is given by
\begin{equation}
\label{equation:zeta}
\big(x:y:z:t:w\big)\mapsto\big(g_9:g_{10}:g_{11}:g_{12}:g_{13}:g_{14}:g_{15}:g_{15}^\prime:g_{16}:g_{17}:g_{18}:g_{19}:g_{20}:g_{21}\big).
\end{equation}
In particular, this reproves \cite[Proposition~4.1]{DinewKapustkaKapustka}.

Denote by $T_i$ and $T_{15}^\prime$ the proper transforms of the surfaces $M_i$ and $M_{15}^\prime$ on the threefold~$V_u$,
respectively. Then
$$
T_i\sim T_{15}^\prime\sim -K_{V_u}\sim H_{V_u}.
$$
This implies that all surfaces $T_i$ and $T_{15}^\prime$ are irreducible,
because the group $\mathrm{Pic}(V_u)$ is generated by the divisor $H_{V_u}$.
This implies that the surface $M_{15}^\prime$ is irreducible, since the surface $T_{15}^\prime$ is irreducible
and $M_{15}^\prime$ does not contain the surface $\mathcal{S}$.
Similarly, the surfaces $M_{10}$ and $M_{20}$ are also irreducible.
However, the remaining surfaces $M_i$ are reducible.
Namely, let $N_3$, $N_5$, $N_8$, $N_{10}$, $N_{13}$ and $N_{15}$
be the surfaces in $Q_u$ that are cut out by
the equations $h_3=0$, $h_5=0$, $h_8=0$, $h_{10}=0$ and~\mbox{$h_{15}=0$}, respectively.
Similarly, let $H_x$, $H_y$, $H_z$, $H_t$ and $H_w$ be the hyperplane sections of the quadric $Q_u$ that are cut out by $x=0$, $y=0$, $z=0$, $t=0$ and $w=0$, respectively.
Then we see from~\eqref{equation:H-3} that
\begin{multline*}
M_9=N_3+\mathcal{S},\quad M_{11}=N_5+\mathcal{S},\quad M_{12}=H_x+2\mathcal{S},\quad M_{13}=H_y+2\mathcal{S}, \\
M_{14}=N_8+\mathcal{S},\quad M_{15}=H_z+2\mathcal{S},\quad M_{16}=N_{10}+\mathcal{S},\quad M_{17}=H_t+2\mathcal{S},\\
M_{18}=H_w+2\mathcal{S},\quad M_{19}=N_{13}+\mathcal{S},\quad M_{21}=N_{15}+\mathcal{S}.
\end{multline*}
Thus, the surfaces $T_9$, $T_{11}$, $T_{14}$, $T_{16}$, $T_{19}$ and $T_{21}$
are actually the proper transforms on the threefold $V_u$ of the surfaces $N_3$, $N_5$, $N_8$, $N_{10}$, $N_{13}$ and $N_{15}$, respectively.
Similarly, the surfaces $T_{12}$, $T_{13}$, $T_{15}$, $T_{17}$ and $T_{18}$ are the proper transforms on the threefold $V_u$
of the surfaces $H_x$, $H_y$, $H_z$, $H_t$ and $H_w$, respectively.

\begin{remark}
\label{remark:C-2-T-i}
It follows from \eqref{equation:zeta} that
the conic $\mathcal{C}_2$ is contained in the surfaces $T_9$, $T_{11}$, $T_{12}$, $T_{13}$, $T_{14}$, $T_{15}$, $T_{16}$, $T_{17}$, $T_{18}$, $T_{19}$ and $T_{21}$,
and it is not contained in the surfaces $T_{10}$, $T_{20}$ and $T_{15}^\prime$.
\end{remark}

\begin{lemma}
\label{lemma:l-1-l-2-T-i}
The line $\ell_1$ is contained in the surfaces
$T_{11}$, $T_{12}$, $T_{13}$, $T_{14}$, $T_{15}$, $T_{15}^\prime$, $T_{16}$, $T_{17}$, $T_{18}$,
$T_{19}$, $T_{20}$, $T_{21}$,
and it is not contained in the surfaces $T_9$ and $T_{10}$.
Similarly, the line $\ell_2$ is contained in the surfaces
$T_{9}$, $T_{10}$, $T_{11}$, $T_{12}$, $T_{13}$, $T_{14}$, $T_{15}$, $T_{15}^\prime$, $T_{16}$, $T_{17}$, $T_{18}$, $T_{19}$,
and it is not contained in the surfaces $T_{20}$ and $T_{21}$.
\end{lemma}

\begin{proof}
Let $P_\lambda\in\mathbb{P}^4$ be the point
$$
\Big(\frac{\lambda(u\lambda-\lambda+1)}{u}:\lambda:\lambda:1:1\Big),
$$
where $\lambda\in\mathbb{C}$.
Let $C$ be the (closure of the) curve swept out by $P_{\lambda}$.
Then $C$ is contained in the quadric $Q_u$,
and
$$
C\cap L_{2}=P_0=\big(0:0:0:1:1\big).
$$
Note that the point $P_0$ is not contained in the curve $\Gamma$, so that the proper transforms
of the curves $C$ and $L_{2}$ on the threefold $\widetilde{Q}_u$ still meet
at the preimage of the point~$P_0$.
This implies that the proper transform $C_{V_u}$ of the curve $C$ on the threefold $V_u$ intersects
the line~$\ell_{2}$.
Substitute the coordinates of the point $P_\lambda$ into \eqref{equation:zeta},
multiply the coordinates of the resulting point by~$\frac{u}{\lambda}$, and let $\lambda=0$.
This gives the point
$$
C_{V_u}\cap\ell_{2}=(0:0:0:0:0:0:0:0:0:0:0:0:1:1-u).
$$
Using the $\mathbb{C}^\ast$-action on $\mathbb{P}^{13}$, we immediately obtain the equations of the line~$\ell_{2}$.
The equations for the line $\ell_{1}$ are obtained in a similar way. Now the required assertion follows from \eqref{equation:zeta}.
\end{proof}

Let us conclude this section by the following lemma.

\begin{lemma}
\label{lemma:points}
There are no $G$-fixed points in $Q_u$ and $V_u$.
\end{lemma}

\begin{proof}
It follows from \eqref{equation:action} that the only $\mathbb{C}^\ast$-fixed points in the quadric $Q_u$
are the points $(1:0:0:0:0)$, $(0:0:0:0:1)$, $(0:1:0:0:0)$ and $(0:0:0:1:0)$.
Note that $\iota$ swaps the points $(1:0:0:0:0)$ and $(0:0:0:0:1)$, and it also swaps the remaining two $\mathbb{C}^\ast$-fixed points,
so that there are no $G$-fixed points in $Q_u$.
This also implies that there are no $G$-fixed points in $\widetilde{Q}_u$.

By Remark~\ref{remark:flop}, the flopping curves of $\chi$ are disjoint and swapped by the involution~$\iota$.
Hence, there are no $G$-fixed points in $\widetilde{V}_u$.
Thus, if $V_u$ contains a $G$-fixed point, then it must be contained in the conic $\mathcal{C}_2$.

Let $\Pi\cong\mathbb{P}^2$ be the linear span of the conic $\mathcal{C}_2$ in $\mathbb{P}^{13}$.
Then $\Pi$ is $G$-invariant.
The action of $G$ on $\Pi$ is not faithful (indeed, it contains all elements of order $5$ in~$\mathbb{C}^\ast$).
However, the kernel is finite, and the automorphism $\iota$ acts faithfully on $\Pi$. This implies that
there is a faithful action of a quotient of $G$ that is
isomorphic to $G$ on $\Pi$ and thus
on $\mathcal{C}_2$.
Therefore, the conic $\mathcal{C}_2$ does not contain $G$-fixed points,
so that there are no $G$-fixed points in $V_u$.
\end{proof}

\section{Invariant curves}
\label{section:curves}

In this section, we make the first steps needed for a description of
irreducible $G$-invariant curves in $Q_u$ and $V_{u}$.
We start with

\begin{lemma}
\label{lemma:parameterization}
Fix a point $(a_0:\ldots:a_n)\in\mathbb{P}^n$, and fix positive integers~\mbox{$r_0\leqslant\ldots\leqslant r_n$}.
Let~$Z$ be the curve in $\mathbb{P}^n$ that is the closure of the subset
$$
\Big\{(\lambda^{r_0}a_0:\ldots:\lambda^{r_n}a_n)\mid \lambda\in\mathbb{C}^\ast\Big\}\subset\mathbb{P}^n.
$$
Denote by $\Sigma$ the set of indices $i$ such that $a_i\neq 0$.
Set
$$
r_k=\min\{r_i\mid i\in\Sigma\}, \quad
r_K=\max\{r_i\mid i\in\Sigma\}.
$$
Denote by $d$ the greatest common divisor of the numbers $r_i-r_{k}$ for $i\in\Sigma$. Then
$$
\deg(Z)=\frac{r_{K}-r_{k}}{d}.
$$
Furthermore, let $s$ be the maximal number of indices $i$ in $\Sigma$ with distinct $r_i$.
Then $Z$ is a rational normal curve if and only if $\deg(Z)=s$.
\end{lemma}

\begin{proof}
Cancelling a common factor in the homogeneous coordinates if necessary,
we may assume that $r_k=0$. To compute the degree of $Z$, note that the
intersection points of $Z$ with a general hyperplane $\Lambda$
in $\mathbb{P}^n$
correspond to the roots of a polynomial $P_{\Lambda}(\lambda)$
of degree $r_K$ in $\lambda$.
Since $P_\Lambda$ is actually a polynomial of degree $r_K/d$ in $\lambda^d$,
the $r_K$ roots of $P_{\Lambda}$  produce $r_K/d$ points of
$\Lambda\cap Z$. Thus, the degree of $Z$ equals $r_K/d$. It remains to notice
that the linear span of $Z$ has dimension $s$, so that $Z$ is a rational normal
curve if and only if $\deg(Z)=s$.
\end{proof}

There are no $G$-fixed points in $Q_u$ by Lemma~\ref{lemma:points}.
This implies, in particular, that  every irreducible $G$-invariant curve in $Q_u$ is rational and contains at least one $\iota$-fixed point.
Hence, every irreducible $G$-invariant curve is a closure of the $\mathbb{C}^\ast$-orbit of any of its  $\iota$-fixed points.

\begin{lemma}
\label{lemma:iota-fixed-points}
All $\iota$-fixed points in $Q_u$ are the points
$$
P_\pm=(1:\pm\sqrt{u}:0:\mp\sqrt{u}:-1)
$$
and the points
\begin{equation}
\label{equation:a-b-c-conic}
\Big(b^2-(1-u)(a-b)^2:u(a^2-b^2)-a^2:a^2-u(a-b)^2:u(a^2-b^2)-a^2:b^2-(1-u)(a-b)^2\Big),
\end{equation}
where $(a:b)\in\mathbb{P}^1$.
\end{lemma}

\begin{proof}
Using \eqref{equation:involution}, one can see that the $\iota$-fixed points in $\mathbb{P}^4$ are
the points of the line
$$
\left\{\aligned
&x+w=0,\\
&y+t=0,\\
&z=0,
\endaligned
\right.
$$
and the points of the plane
$$
\left\{\aligned
&x-w=0,\\
&y-t=0.
\endaligned
\right.
$$
Intersecting the line with $Q_u$, we obtain the points $P_\pm$.
Similarly, intersecting the plane with the quadric $Q_u$, we obtain the conic parameterized by \eqref{equation:a-b-c-conic}.
\end{proof}

Observe that the $\mathbb{C}^\ast$-orbit of the point $P_+$ is the same as the $\mathbb{C}^\ast$-orbit of the point~$P_-$.
We~denote its closure by $\Theta_{\pm}$. Similarly, we denote the closure of the $\mathbb{C}^\ast$-orbit
of the point~\eqref{equation:a-b-c-conic} by $\Theta_{a,b}$.
By construction, the curves $\Theta_{\pm}$ and $\Theta_{a,b}$ are all irreducible $G$-invariant curves contained in the quadric $Q_u$.

\begin{lemma}
\label{lemma:curves-in-S4}
The only irreducible $G$-invariant curves in $\mathcal{S}$ are
$$
\Gamma=\Theta_{0,1}=\Theta_{u,u-1}
$$
and $\Theta_{1,0}=\Theta_{1,1}$. The degree of the curve $\gamma(\Theta_{1,0})$ in $\mathbb{P}^{10}$ is $12$.
\end{lemma}

\begin{proof}
Recall from \S\ref{section:explicit} that the surface $\mathcal{S}$ is cut out on the quadric $Q_u$ by the equation~\mbox{$f=0$}, where $f=xw-yt$.
Substituting $x=1$, $y=\pm\sqrt{u}$, $z=0$, $t=\mp\sqrt{u}$ and~\mbox{$w=-1$} into the polynomial~$f$,
we get $u-1$, so that the curve $\Theta_{\pm}$ is not contained in $\mathcal{S}$.
Similarly, substituting the coordinates of the point \eqref{equation:a-b-c-conic} into $f$,
we obtain
$$
4(1-u)ab(a-b)(u(a-b)-a),
$$
and the first assertion follows.

The curve $\Theta_{1,0}$ is the closure of the $\mathbb{C}^\ast$-orbit of the point $P=(1:1:-1:1:1)$.
Thus, by \eqref{equation:gamma}, the curve $\gamma(\Theta_{1,0})$ is the closure of the $\mathbb{C}^\ast$-orbit of the point
$$
\gamma(P)=(1:1:0:0:1:0:1:0:0:1:1),
$$
so that the degree of the curve $\gamma(\Theta_{0,1})$ is $12$ by Lemma~\ref{lemma:parameterization}.
\end{proof}

Let $\Delta$ be the conic in $Q_u$ that is cut out by
\begin{equation}
\label{equation:Delta}
y=t=0.
\end{equation}
Then $\Delta$ is $G$-invariant. One can check that
$$
\Delta=\Theta_{\sqrt{u},\sqrt{u-1}}=\Theta_{-\sqrt{u},\sqrt{u-1}}.
$$
Similarly, let $\Upsilon$ be the conic in $Q_u$ that is cut out by
\begin{equation}
\label{equation:Upsilon}
x=w=0
\end{equation}
Then $\Upsilon$ is $G$-invariant. One can check that
$$
\Upsilon=\Theta_{\sqrt{1-u}+1,\sqrt{1-u}}=\Theta_{\sqrt{1-u}-1,\sqrt{1-u}}.
$$

\begin{lemma}
\label{lemma:curves-in-V22-degree-10-12}
The following assertions hold.
\begin{itemize}
\item[(i)]  The curve $\zeta(\Theta_{\pm})$ is a curve of degree~$12$. One has~\mbox{$\zeta(\Theta_{\pm})\subset T_{15}\cap T_{15}^\prime$}.

\item[(ii)] The curve $\zeta(\Delta)$ is a rational normal curve of degree~$4$. One has $\zeta(\Delta)\subset T_{10}\cap T_{20}$.

\item[(iii)] The curve $\zeta(\Upsilon)$ is a rational normal curve of degree~$6$. One has $\zeta(\Upsilon)\subset T_{10}\cap T_{20}$.

\item[(iv)] For every curve $\Theta_{a,b}$ not contained in the surface $\mathcal{S}$ and different from $\Delta$ and $\Upsilon$, the degree of $\zeta(\Theta_{a,b})$ is either $10$ or $12$.

\item[(v)] If $\Theta_{a,b}$ is not contained in the surface $\mathcal{S}$, then the degree of the curve $\zeta(\Theta_{a,b})$ equals~$10$ if and only if the curve $\Theta_{a,b}$ is contained in $N_{3}\cap N_{15}$.
\end{itemize}
\end{lemma}

\begin{proof}
By \eqref{equation:zeta}, the curve $\zeta(\Theta_{\pm})$ is the closure of the $\mathbb{C}^\ast$-orbit of the point $\zeta(P_{+})$ that is
$$
\Big(u\sqrt{u}:-u:-\sqrt{u}:u-1:\sqrt{u}(u-1):-u:0:0:u:-\sqrt{u}(u-1):-u+1:\sqrt{u}:u:-u\sqrt{u}\Big),
$$
which is contained in $T_{15}\cap T_{15}^\prime$.
Then $\zeta(\Theta_{\pm})$ is a curve of degree~$12$ by Lemma~\ref{lemma:parameterization},
and it is contained in $T_{15}\cap T_{15}^\prime$.
This proves assertion (i).

To prove assertions (ii), (iii) and (iv),
we need some auxiliary computations.
Define the polynomial
$$
q_0=(u-1)^2a^4-2(u-1)^2a^3b+2(u-1)(u-2)a^2b^2-6u(u-1)ab^3+u(3u-2)b^4.
$$
Furthermore, define the polynomials
$$
\aligned
&q_1=(u-1)a^2-ub^2,\\
&q_2=(u-1)a^2-(2u-2)ab+ub^2,\\
&q_3=(u-1)a^2+2ab-(u+2)b^2,\\
&q_4=(u-1)a^2-(2u-2)ab+(u-2)b^2,\\
&q_5=(u-1)a^2-2uab+ub^2,\\
&q_6=(u-1)a^2-(2u-4)ab+(u-4)b^2.
\endaligned
$$
Recall that $u\ne 0$ and $u\ne 1$.
Observe that $q_i$ is coprime to $q_j$ for $0\leqslant i<j\leqslant 6$ with the following exceptions:
\begin{itemize}
\item $q_0$ is divisible by $q_6$ provided that $u^2-2u+2=0$;

\item $q_1=q_6$ provided that $u=2$;

\item $q_3=q_5$ provided that $u=-1$;

\item $q_2$ and $q_3$ have a common linear factor provided that $u=\frac{-1\pm\sqrt{5}}{2}$.
\end{itemize}

Substituting the coordinates of the point \eqref{equation:a-b-c-conic} into the polynomials $g_i$ and $g_{15}^\prime$,
we obtain the polynomials $p_i$ and $p_{15}^\prime$ (in $a$ and $b$), respectively.
We compute
$$
\aligned
&p_{9}=p_{21}=-8(u-1)a^2b(a-b)((u-1)a-ub)^2q_0,\\
&p_{10}=p_{20}=4a^2((u-1)a-ub)^2q_1q_2q_3,\\
&p_{11}=p_{19}=-8(u-1)a^2b(a-b)((u-1)a-ub)^2q_1q_4,\\
&p_{12}=p_{18}=16(u-1)^2a^2b^2(a-b)^2((u-1)a-ub)^2q_2,\\
&p_{13}=p_{17}=16(u-1)^2a^2b^2(a-b)^2((u-1)a-ub)^2q_1,\\
&p_{14}=p_{16}=-8(u-1)a^2b(a-b)((u-1)a-ub)^2q_1q_2,\\
&p_{15}=-16(u-1)^2a^2b^2(a-b)^2((u-1)a-ub)^2q_5,\\
&p_{15}^\prime=4(u-1)a^2((u-1)a-ub)^2q_1^2q_6.\\
\endaligned
$$

Let us consider the curve $\Theta_{a,b}$ not contained
in the surface $\mathcal{S}$. By Lemma~\ref{lemma:curves-in-S4}
this means that $a\neq 0$, $b\neq 0$, $a-b\neq 0$ and
$(u-1)a-ub\neq 0$.
These conditions imply that
\begin{itemize}
\item
the polynomials $p_{9}$ and $p_{21}$ vanish if and only if $q_0$ does,
\item
the polynomials $p_{10}$ and $p_{20}$ vanish if and only if either $q_1$, or $q_2$,
or $q_3$ does,
\item
the polynomials $p_{11}$ and $p_{19}$ vanish if and only if either $q_1$ or $q_4$ does,
\item
the polynomials $p_{12}$ and $p_{18}$ vanish if and only if $q_2$ does,
\item
the polynomials $p_{13}$ and $p_{17}$ vanish if and only if $q_1$ does,
\item
the polynomials $p_{14}$ and $p_{16}$ vanish if and only if either $q_1$ or $q_2$ does,
\item
the polynomial $p_{15}$ vanishes if and only if $q_5$ does,
\item
the polynomial $p_{15}^\prime$ vanishes if and only if either $q_1$ or $q_6$ does.
\end{itemize}
Note that $q_1=0$ if and only if $\Theta_{a,b}=\Delta$,
and $q_2=0$ if and only if $\Theta_{a,b}=\Upsilon$.

Suppose that $\Theta_{a,b}=\Delta$. Then $q_1=0$, so that
\begin{equation}
\label{equation:Delta-polynomials}
p_{10}=p_{11}=p_{13}=p_{14}=p_{15}^\prime=p_{16}=p_{17}=p_{19}=p_{20}=0.
\end{equation}
The coprimeness properties of the polynomials $q_i$ imply that
$p_9$, $p_{12}$, $p_{15}$, $p_{18}$ and $p_{21}$ do not vanish.
Therefore, $\zeta(\Delta)$ is a rational normal curve of degree~$4$
by~\eqref{equation:zeta} and Lemma~\ref{lemma:parameterization},
which proves assertion (ii).

Suppose that $\Theta_{a,b}=\Upsilon$. Then $q_2=0$, so that
\begin{equation}
\label{equation:Upsilon-polynomials}
p_{10}=p_{12}=p_{14}=p_{16}=p_{18}=p_{20}=0.
\end{equation}
The coprimeness properties of the polynomials $q_i$ imply that
$p_9$, $p_{11}$, $p_{13}$, $p_{15}$, $p_{17}$, $p_{19}$ and~$p_{21}$ do not vanish.
Therefore, we see that $\zeta(\Upsilon)$ is a rational normal curve of degree~$6$
by~\eqref{equation:zeta} and Lemma~\ref{lemma:parameterization},
which proves assertion (iii).

Now suppose that $\Theta_{a,b}$ is different from $\Delta$ and $\Upsilon$.
This means that $q_1\neq 0$ and $q_2\neq 0$, so that
in particular $p_{12}$ and $p_{13}$ do not vanish.
If $q_0\neq 0$, then $p_9$ and $p_{21}$ do not vanish as well,
so that the degree of the curve $\zeta(\Theta_{a,b})$ is~$12$
by~\eqref{equation:zeta} and Lemma~\ref{lemma:parameterization}.
Thus, we may assume that $q_0=0$, so that
$$
p_9=p_{21}=0.
$$
The coprimeness properties of the polynomials $q_i$ imply that
$p_{10}$, $p_{11}$ and $p_{20}$ do not vanish,
so that the degree of the curve $\zeta(\Theta_{a,b})$ is~$10$ by~\eqref{equation:zeta} and Lemma~\ref{lemma:parameterization}.
This proves assertion (iv).
The condition $p_9=p_{21}=0$ means that the curve $\Theta_{a,b}$ is contained in~$M_{9}$ and~$M_{21}$.
Since $M_9=N_3+\mathcal{S}$ and $M_{21}=N_{15}+\mathcal{S}$, we see that $\Theta_{a,b}$ is contained in~$N_{3}$ and~$N_{15}$,
because we assume that $\Theta_{a,b}$ is not contained in~$\mathcal{S}$.
This proves assertion~(v) and completes the proof of the lemma.
\end{proof}

Taking a more careful look at the proof of Lemma~\ref{lemma:curves-in-V22-degree-10-12},
one can deduce that there are only a finite number of curves among $\zeta(\Theta_{a,b})$ that are \emph{not} rational normal curves of degree~$12$.
Moreover, one can explicitly describe all such curves for any given~$u$.

\begin{remark}
\label{remark:pencil-base-locus-easy}
By Lemma~\ref{lemma:curves-in-V22-degree-10-12}(i),
the intersection $T_{15}\cap T_{15}^\prime$ contains the curve $\zeta(\Theta_{\pm})$,
which is a curve of degree $12$.
Moreover, it follows from Lemma~\ref{lemma:l-1-l-2-T-i} that
$T_{15}\cap T_{15}^\prime$ contains both lines $\ell_1$ and $\ell_2$.
Thus, the intersection $T_{15}\cap T_{15}^\prime$ does not contain irreducible $G$-invariant curves
of degree greater than $8$ that are different from the curve~\mbox{$\zeta(\Theta_{\pm})$}.
Note that $T_{15}\cap T_{15}^\prime$ does not contain the conic $\mathcal{C}_2$ by Remark~\ref{remark:C-2-T-i}.
Using~\eqref{equation:Delta},
we see that $T_{15}\cap T_{15}^\prime$ does not contain the curve $\mathcal{C}_4$.
Similarly, using \eqref{equation:Upsilon},
we see that~\mbox{$T_{15}\cap T_{15}^\prime$} does not contain the curve $\mathcal{C}_6$.
\end{remark}

Let us describe explicitly the curves $\Theta_{a,b}$
in the case when $\zeta(\Theta_{a,b})$ is a curve of degree~$10$.
If $u\ne-\frac{1}{3}$, let $\vartheta$ be one of the roots $\sqrt{(3u+1)(1-u)}$.
If $u=-\frac{1}{3}$, let $\vartheta=0$.
If~\mbox{$u=\frac{2}{3}$}, then
$$
(3u+1)(1-u)=1.
$$
In this case, we assume that $\vartheta=1$.
Observe that the quadric $Q_u$ contains the point
\begin{equation}
\label{equation:C10-point}
\Big(1:1:1:\frac{(u-1)(\vartheta-u-1)}{2u^2}:\frac{(u-1)(2u^2+\vartheta-u-1)}{2u^3}\Big).
\end{equation}
Similarly, the quadric $Q_u$ contains the point
\begin{equation}
\label{equation:C10-prime-point}
\Big(1:1:1:\frac{(u-1)(-\vartheta-u-1)}{2u^2}:\frac{(u-1)(2u^2-\vartheta-u-1)}{2u^3}\Big).
\end{equation}
Let $\Psi$ be the closure of the $\mathbb{C}^\ast$-orbit of the point \eqref{equation:C10-point},
and let $\Psi^\prime$ be the closure of the $\mathbb{C}^\ast$-orbit
of the point \eqref{equation:C10-prime-point}.
Then the curve $\Psi$ is $G$-invariant, since
the $\mathbb{C}^\ast$-orbit of the point \eqref{equation:C10-point} contains the image of this point via the involution $\iota$,
because
\begin{multline*}
\Big(1:\lambda:\lambda^3:\lambda^5\frac{(u-1)(\vartheta-u-1)}{2u^2}:\lambda^6\frac{(u-1)(2u^2+\vartheta-u-1)}{2u^3}\Big)=\\
=\Big(\frac{(u-1)(2u^2+\vartheta-u-1)}{2u^3}:\frac{(u-1)(\vartheta-u-1)}{2u^2}:1:1:1\Big)
\end{multline*}
for $\lambda=\frac{u(\vartheta-u-1)}{(2u^2+\vartheta-u-1)}\in\mathbb{C}^\ast$.
Similarly, we see that the curve $\Psi^\prime$ is $G$-invariant. Of course, the curves
$\Psi$ and $\Psi^\prime$ are of the form $\Theta_{a,b}$ for certain $a$ and $b$, but we will
never use the values of these parameters.

It is straightforward to check that
$\Psi=\Psi^\prime$ if and only if $u=-\frac{1}{3}$.
Moreover, if $u=\frac{2}{3}$, then $\Psi\ne\Gamma$ and $\Psi^\prime=\Gamma$.
This explains why we let $\vartheta=1$ in this case.

\begin{lemma}
\label{lemma:curves-in-V22-degree-10}
The following assertions hold.
\begin{itemize}
\item[(i)] Both curves $\Psi$ and $\Psi^\prime$ are contained in the intersection $N_3\cap N_{15}$.

\item[(ii)] The curve $\Psi$ is not contained in $\mathcal{S}$. If $u\ne\frac{2}{3}$, then $\Psi^\prime$ is not contained in $\mathcal{S}$.

\item[(iii)] The curve $\zeta(\Psi)$ is a curve of degree~$10$.

\item[(iv)] If $u\ne\frac{2}{3}$, then $\zeta(\Psi^\prime)$ is a curve of degree~$10$.

\item[(v)] If $\Theta_{a,b}\not\subset\mathcal{S}$ and $\zeta(\Theta_{a,b})$ is a curve of degree $10$, then $\Theta_{a,b}=\Psi$ or $\Theta_{a,b}=\Psi^\prime$.

\item[(vi)] The surfaces $N_3$ and $N_{15}$ are tangent along $\Gamma$
if and only if $u=\frac{2}{3}$.

\item[(vii)] If $u=\frac{2}{3}$, then $N_3$ and $N_{15}$ do not tangent $\mathcal{S}$ at a general point of the curve $\Gamma$.

\item[(viii)] If $u=-\frac{1}{3}$, then $N_3$ and $N_{15}$ are tangent along $\Psi=\Psi^\prime$.
\end{itemize}
\end{lemma}

\begin{proof}
Using \eqref{equation:quadric}, we see that the intersection $N_3\cap N_{15}$ is given in $\mathbb{P}^4$ by
\begin{equation}
\label{equation:two-cubics}
\left\{\aligned
&y^3-x^2z=0,\\
&t^3-zw^2=0,\\
&u(xw-z^2)+(z^2-yt)=0.\\
\endaligned
\right.
\end{equation}
In fact, this system of equation defines an effective one-cycle in $Q_u$ of degree $18$, which contains the curve $\Gamma$.

Let us show that $N_3\cap N_{15}$ contains the curves $\Psi$ and $\Psi^\prime$.
To do this, we may consider the subset where~\mbox{$x\ne 0$}, so that we let $x=1$.
Substituting $z=y^3$ and
$$
w=\frac{yt}{u}+\frac{u-1}{u}z^2
$$
into $t^3-zw^2=0$,
we obtain the equation
$$
\big(t-y^5\big)\Big(t^2u^2+(u^2-1)ty^5+(u-1)^2y^{10}\Big)=0.
$$
If $t=y^5$, we get the curve $\Gamma$.
Thus, the remaining part of the subset \eqref{equation:two-cubics} consists of the $\mathbb{C}^\ast$-orbits of the points
$$
\Big(1:1:1:t:\frac{t+u-1}{u}\Big)
$$
where $t$ is a solution of the quadratic equation
\begin{equation*}
u^2t^2+(u^2-1)t+(u-1)^2=0.
\end{equation*}
Solving this equation, we obtain exactly the points \eqref{equation:C10-point} and \eqref{equation:C10-prime-point}.
This shows that~\eqref{equation:two-cubics} contains the curves $\Psi$ and $\Psi^\prime$.
This proves assertion (i).

Observe that the intersection $\mathcal{S}\cap N_3$ consists of the curve $\Gamma$, the line $L_{2}$, and the line~\mbox{$y=z=w=0$}.
Similarly, the intersection $\mathcal{S}\cap N_{15}$ consists of the curve $\Gamma$, the line~$L_{1}$, and the line $x=z=t=0$.
Thus, the curve $\Psi$ is contained in $\mathcal{S}$ if and only if $\Psi=\Gamma$.
Since $\mathcal{S}$ is cut out on $Q_u$ by the equation $xw=yt$, we see that if $\Psi$ is contained in $\mathcal{S}$, then
$$
\frac{(u-1)(\vartheta-u-1)}{2u^2}=\frac{(u-1)(2u^2+\vartheta-u-1)}{2u^3}.
$$
Simplifying this equation, we get $\vartheta=\frac{3u^2-1}{u-1}$, which implies that $u=\frac{2}{3}$,
so that $\vartheta=1$ by assumption, which implies that the point~\eqref{equation:C10-point} is not contained in $\mathcal{S}$.
Hence, we see that~$\Psi$ is not contained in $\mathcal{S}$.
Similarly, we see that $\Psi^\prime$ is contained in $\mathcal{S}$ if and only if~\mbox{$u=\frac{2}{3}$}. This proves assertion (ii).

Since $\Psi$ is not contained in $\mathcal{S}$, we see that $\zeta(\Psi)$ is a curve of degree~$10$ by Lemma~\ref{lemma:curves-in-V22-degree-10-12}(v).
Similarly, if $u\ne\frac{2}{3}$, then $\Psi^\prime$ is not contained in $\mathcal{S}$, so that $\zeta(\Psi^\prime)$ is a curve of degree~$10$ by Lemma~\ref{lemma:curves-in-V22-degree-10-12}(v) as well.
This proves assertions (iii) and (iv).

If $\Theta_{a,b}$ is not contained in the surface $\mathcal{S}$ and $\zeta(\Theta_{a,b})$ is a curve of degree $10$,
then $\Theta_{a,b}$ is contained in $N_3\cap N_{15}$ by Lemma~\ref{lemma:curves-in-V22-degree-10-12}(v).
On the other hand, the intersection $N_3\cap N_{15}$ is given by  \eqref{equation:two-cubics}.
We just proved that this system of equation defines the union $\Gamma\cup\Psi\cup\Psi^\prime$,
so that either $\Theta_{a,b}=\Psi$ or $\Theta_{a,b}=\Psi^\prime$.
This proves assertion (v).

To prove assertions (vi) and (vii),
let us find the local equations of the surfaces $N_3$, $N_{15}$ and $\mathcal{S}$ at the point $(1:1:1:1:1)$.
We may work in a chart $x\ne 0$, so that we let~\mbox{$x=1$}.
Substituting $w=\frac{yt}{u}+\frac{u-1}{u}z^2$
into the equation $t^3-w^2z=0$ and multiplying the resulting equation by~$u^2$, we obtain the equation
$$
t^3u^2-t^2y^2z+2(1-u)tyz^3-(u-1)^2z^5=0.
$$
Similarly, the surface~$\mathcal{S}$ is given by $ty=z^2$, and the surface $N_3$ is given by $z=y^3$.
Now introducing new coordinates
$\bar{y}=y-1$, $\bar{z}=z-1$ and $\bar{t}=t-1$, we see that $N_{15}$ is given by
$$
2\bar{y}+(5u-4)\bar{z}+(2-3u)\bar{t}+\text{higher order terms}=0.
$$
Similarly, the surface $\mathcal{S}$ is given by
\begin{equation}
\label{equation:S-local}
\bar{y}-2\bar{z}+\bar{t}+\text{higher order terms}=0,
\end{equation}
while the linear term of the defining equation of the surface $N_3$ is $3\bar{y}-\bar{z}$.
Hence, the surface $N_3$ is not tangent to $\mathcal{S}$ at the point $(1:1:1:1:1)$.
Similarly, we see that the surface $N_3$ is tangent to $N_{15}$ at the point $(1:1:1:1:1)$ if and only if $u=\frac{2}{3}$.
This proves assertions (vi) and (vii).

To prove assertion (viii), we assume that $u=-\frac{1}{3}$.
Then $\Psi=\Psi^\prime$, and the point \eqref{equation:C10-point} is the point $(1:1:1:4:-8)$.
Arguing as above, we see that the local equations of the surfaces $N_3$ and $N_{15}$ at the point $(1:1:1:4:-8)$
have the same linear part (in coordinates $\bar{y}=y-1$, $\bar{z}=z-1$ and $\bar{t}=t-4$).
Hence, the surface $N_3$ is tangent to $N_{15}$ at the point $(1:1:1:4:-8)$.
This proves assertion~(viii) and completes the proof of the lemma.
\end{proof}

Recall from Remark~\ref{remark:flop} that the birational map $\zeta$ in \eqref{equation:V22} induces an isomorphism
$$
Q_v\setminus\mathcal{S}\cong V_u\setminus\mathcal{R}.
$$
Therefore, from \eqref{equation:zeta} and Lemmas~\ref{lemma:curves-in-V22-degree-10-12} and \ref{lemma:curves-in-V22-degree-10},
we obtain an explicit description of all irreducible $G$-invariant curves in the Fano threefold $V_u$ that are not contained in the surface~$\mathcal{R}$.
Thus, to classify all such curves in $V_u$, we need to describe those of them
that are contained in~$\mathcal{R}$. This will be done in the next section.

\section{Invariant curves in the surface $\mathcal{R}$}
\label{section:curves-2}

In this section we describe irreducible
$G$-invariant curves in the surface~$\mathcal{R}$, and complete the classification of
irreducible $G$-invariant curves in the threefold~$V_u$
(see Proposition~\ref{proposition:curves}).
We will show that $\mathcal{R}$ contains exactly two irreducible $G$-invariant curves,
one of which is the conic $\mathcal{C}_2$.
To describe the other curve, we analyze all irreducible $G$-invariant curves in surface $E_{Q_u}$.
We start with

\begin{remark}
\label{remark:flop-Gamma}
Recall from Remark~\ref{remark:dP4} that the surface $\mathcal{S}$ is smooth at every point of the curve $\Gamma$
except for the points  $(1:0:0:0:0)$ and $(0:0:0:0:1)$, which are isolated ordinary double singularities.
This implies that
$$
\widetilde{\mathcal{S}}\big\vert_{E_{Q_u}}=\widetilde{\Gamma}+\textbf{l}_1+\textbf{l}_2
$$
for some section $\widetilde{\Gamma}$ of the projection  $E_{Q_u}\to\Gamma$,
where $\textbf{l}_1$ and $\textbf{l}_2$ are the fibers of this projection over the points  $(1:0:0:0:0)$ and $(0:0:0:0:1)$, respectively.
The curve~$\widetilde{\Gamma}$ is irreducible and $G$-invariant.
Since $\widetilde{\Gamma}$ is contained in $\widetilde{\mathcal{S}}$, its image in $V_u$ is the conic $\mathcal{C}_2$.
\end{remark}

Now let us show that $E_{Q_u}$ contains exactly two irreducible $G$-invariant curves.

\begin{lemma}
\label{lemma:E-Q-u-two-curves}
The surface $E_{Q_u}$ contains exactly two irreducible $G$-invariant curves.
One of them is the curve $\widetilde{\Gamma}$ from Remark~\ref{remark:flop-Gamma}.
The second one is also a section of the projection~\mbox{$E_{Q_u}\to\Gamma$}.
\end{lemma}

\begin{proof}
Let $\textbf{l}$ be the fiber of the natural projection $E_{Q_u}\to\Gamma$ over the point $(1:1:1:1:1)$.
Then $\textbf{l}\cong\mathbb{P}^1$ and the curve $\textbf{l}$ is $\iota$-invariant.
Thus, either $\iota$ fixes every point in $\textbf{l}$, or $\iota$ fixes exactly two points in $\textbf{l}$.
Let us show that the former case is impossible.
To do this, recall from \S\ref{section:explicit} that
$$
\Gamma\subset N_3\cap N_{5}\cap N_{8}\cap N_{10}\cap N_{13}\cap N_{15},
$$
and the surfaces $N_3$, $N_5$, $N_8$, $N_{10}$, $N_{13}$, $N_{15}$ are smooth at a general point of the curve~$\Gamma$.
Denote by $\widetilde{N}_3$, $\widetilde{N}_5$, $\widetilde{N}_8$, $\widetilde{N}_{10}$, $\widetilde{N}_{13}$ and $\widetilde{N}_{15}$
the proper transforms of the surfaces $N_3$, $N_5$, $N_8$, $N_{10}$, $N_{13}$ and $N_{15}$ on the threefold $\widetilde{Q}_u$, respectively.
Then each intersection
$$
\widetilde{N}_3\cap\textbf{l},\quad \widetilde{N}_5\cap\textbf{l},\quad \widetilde{N}_8\cap\textbf{l},\quad \widetilde{N}_{10}\cap\textbf{l},\quad \widetilde{N}_{13}\cap\textbf{l},\quad \widetilde{N}_{15}\cap\textbf{l}
$$
consists of a single point.
Moreover, if $u\ne\frac{2}{3}$, then $N_3$ is not tangent to $N_{15}$ at a general point of $\Gamma$ by Lemma~\ref{lemma:curves-in-V22-degree-10}(vi).
Hence, in this case, we have
$$
\widetilde{N}_3\cap\textbf{l}\ne\widetilde{N}_{15}\cap\textbf{l},
$$
so that the involution $\iota$ swaps these two points, since $\iota(N_3)=N_{15}$.
Thus, if $u\ne\frac{2}{3}$, then the involution $\iota$ acts on the curve $\textbf{l}$ non-trivially.

Recall that $\iota(N_5)=N_{13}$, the surface $N_5$ is cut out on $Q_u$ by $x^2t-y^2z=0$,
and the surface $N_5$ is cut out on $Q_u$ by $yw^2-zt^2=0$.
Let us find out when $N_5$ is tangent to~$N_{13}$ at a general point of $\Gamma$.
To do this, let us describe the local equations of the surfaces~$N_5$ and~$N_{13}$ at the point $(1:1:1:1:1)$.
We may work in a chart $x\ne 0$, so that we let~\mbox{$x=1$}.
Substituting
$$
w=\frac{yt}{u}+\frac{u-1}{u}z^2
$$
into $yw^2-zt^2=0$ and multiplying the resulting equation by $u^2$, we obtain the equation
$$
t^2y^3-u^2t^2z+2(u-1)ty^2z^2+(u-1)^2yz^4=0.
$$
This is the equation of $N_{13}$.
The equation of the surface $N_5$ is simply $t=y^2z$.
Now introducing new coordinates
$\bar{y}=y-1$, $\bar{z}=z-1$ and $\bar{t}=t-1$, we see that $N_{13}$ is given by
$$
(u+2)\bar{y}+(3u-4)\bar{z}+2(1-u)\bar{t}+\text{higher order terms}=0.
$$
Similarly, the surface $N_{13}$ is given by
$$
2\bar{y}+\bar{z}-\bar{t}+\text{higher order terms}=0.
$$
This implies that $N_5$ is tangent to $N_{13}$ at the point $(1:1:1:1:1)$ if and only if $u=2$.

Recall from Lemma~\ref{lemma:curves-in-V22-degree-10}(vi) that $N_3$ is tangent to $N_{15}$
at a general point of the curve $\Gamma$ if and only if $u=\frac{2}{3}$.
We see that $N_5$ is tangent to the surface $N_{13}$ at a general point of the curve $\Gamma$ if and only if $u=2$.
The same arguments imply that $N_8$ is never tangent to $N_{10}$ at a general point of the curve $\Gamma$.
Arguing as above, we see that $\iota$ acts on $\textbf{l}$ non-trivially as claimed.

Since $\iota$ acts non-trivially on the fiber $\textbf{l}$, it fixes two points in $\textbf{l}$.
One of them is the point $\textbf{l}\cap\widetilde{\mathcal{S}}$.
It is contained in $\widetilde{\Gamma}$, so that $\widetilde{\Gamma}$ is the closure of the $\mathbb{C}^\ast$-orbit of the point~$\textbf{l}\cap\widetilde{\mathcal{S}}$.
Similarly, the closure of the $\mathbb{C}^\ast$-orbit of the second fixed point of the involution $\iota$ is another irreducible $G$-invariant curve in $E_{Q_u}$.
Then every irreducible $G$-invariant curve in $E_{Q_u}$ must be one of these two curves.
Indeed, an irreducible $G$-invariant curve in $E_{Q_u}$ cannot be contracted by $\pi$, since $Q_u$ does not have $G$-fixed points.
Moreover,
since all $\mathbb{C}^\ast$-orbits in $E_{Q_u}$ that are not contained in
the fibers of the projection $E_{Q_u}\to\Gamma$ are its sections,
we conclude that an intersection of any irreducible $G$-invariant curve in $E_{Q_u}$
with~$\textbf{l}$ must consist of a $\iota$-invariant point,
which in turn uniquely determines this curve.
Since we proved that $\textbf{l}$ contains exactly two $\iota$-fixed points,
an irreducible $G$-invariant curve in~$E_{Q_u}$ must be the closure of the $\mathbb{C}^\ast$-orbit of one of these two points.
This completes the proof of the lemma.
\end{proof}

Thus, the surface $E_{Q_u}$ contains exactly two irreducible $G$-invariant curves.
One of them is the curve $\widetilde{\Gamma}$ from Remark~\ref{remark:flop-Gamma}.
The second curve can be described rather explicitly.

\begin{remark}
\label{remark:E-Q-u-two-curves}
Let us use the notation of the proof of Lemma~\ref{lemma:E-Q-u-two-curves}.
Recall from this proof that $\iota$ fixes exactly two points in $\textbf{l}$.
One of them is the point $\textbf{l}\cap\widetilde{\mathcal{S}}$.
To describe the second $\iota$-fixed point in $\textbf{l}$, denote by $M_{15}^\mu$ the surface in $Q_u$ that is cut out by the equation
$$
g_{15}^\prime+\mu g_{15}=0,
$$
where $\mu\in\mathbb{C}$.
Denote by $\widetilde{M}_{15}^\mu$ the proper transform of the surface $M_{15}^\mu$ on the threefold~$\widetilde{Q}_u$.
Then $M_{15}^\mu$ is singular along $\Gamma$ by Lemma~\ref{lemma:quintics}.
Moreover, it has a double point at a general point of $\Gamma$.
To determine its type, let us describe the local equation of the surface $M_{15}^\mu$ at the point $(1:1:1:1:1)$.
We may work in the chart~\mbox{$x\ne 0$}, so that we let~\mbox{$x=1$}.
Substituting~\mbox{$x=1$} and $w=\frac{yt}{u}+\frac{u-1}{u}z^2$ into $g_{15}^\prime+\mu g_{15}$ and multiplying the result by $u^2$, we obtain the polynomial
\begin{multline*}
u^2t^3+t^2y^5+(u^2\mu-2u\mu+\mu+u-4)t^2y^2z+\\
+2(u-1)ty^4z^2+(8-2u^2\mu+4u\mu-3u^2-2\mu-4u)tyz^3+\\
+(u-1)^2y^3z^4+(u^2\mu-2u\mu+u^2+\mu+3u-4)z^5.
\end{multline*}
Then introducing new coordinates
$\bar{y}=y-1$, $\bar{z}=z-1$ and $\bar{t}=t-1$, we rewrite this polynomial as
\begin{multline}
\label{equation:quadratic-part}
(\mu u^2-2\mu u+3u^2+\mu+u-3)\bar{t}^2+\\
+(2\mu u^2-4\mu u-3u^2+2\mu+8u-6)\bar{t}\bar{y}+(12-4\mu u^2+8\mu u-9u^2-4\mu-6u)\bar{t}\bar{z}+\\
+(\mu u^2-2\mu u+3u^2+\mu+7u-3)\bar{y}^2+(12-4\mu u^2+8\mu u+3u^2-4\mu-18u)\bar{y}\bar{z}+\\
+(4\mu u^2-8\mu u+7u^2+4\mu+8u-12)\bar{z}^2+\text{higher order terms}.
\end{multline}
If $\mu\ne-\frac{3u^2+16u-16}{4(u-1)^2}$, then the surface $M_{15}^\mu$ has an non-isolated ordinary double point at a general point of $\Gamma$.
Vice versa, if $\mu=-\frac{3u^2+16u-16}{4(u-1)^2}$, then the quadratic part of the polynomial~\eqref{equation:quadratic-part} simplifies as
$$
\frac{1}{4}\Big((2+3u)\bar{y}+4(u-1)\bar{z}+(2-3u)\bar{t}\Big)^2.
$$
Comparing it with~\eqref{equation:S-local}, we see that the intersection
$\widetilde{M}_{15}^\mu\cap\textbf{l}$ consists of a single point that is not contained in $\widetilde{\mathcal{S}}$.
This is the second point fixed in $\textbf{l}$ by the involution $\iota$.
\end{remark}

\begin{remark}
\label{remark:E-Q-u-2-3}
Suppose that $u=\frac{2}{3}$.
Let $\widetilde{Z}$ be an irreducible $G$-invariant curve contained in the surface $E_{Q_u}$ that is different from the curve $\widetilde{\Gamma}$.
Denote by $\widetilde{\Psi}$ the proper transform of the curve $\Psi$ on the threefold~$\widetilde{Q}_u$.
Let us use the notation from the proof of Lemma~\ref{lemma:E-Q-u-two-curves} and Remark~\ref{remark:E-Q-u-two-curves}.
Then
$$
\widetilde{N}_3\cap\widetilde{N}_{15}=\widetilde{Z}\cup\widetilde{\Psi}
$$
by Lemma~\ref{lemma:curves-in-V22-degree-10}(vi),
because $N_3$ is smooth at~the point \mbox{$(1:0:0:0:0)$}, and $N_{15}$ is smooth at~the point \mbox{$(0:0:0:0:1)$}.
Observe also that the curve $\widetilde{L}_{1}$ is contained in $\widetilde{N}_3$, and it is not contained in $\widetilde{N}_{15}$.
Similarly, the curve $\widetilde{L}_{2}$ is contained in $\widetilde{N}_{15}$, and it is not contained in~$\widetilde{N}_{3}$.
Thus, since $\widetilde{N}_{15}\cdot\widetilde{L}_{1}=0$ and $\widetilde{N}_{3}\cdot\widetilde{L}_{2}=0$,
we see that $\widetilde{L}_{1}$ is disjoint from $\widetilde{N}_{15}$,
and~$\widetilde{L}_{2}$ is disjoint from $\widetilde{N}_{3}$.
Using \eqref{equation:V22} and \eqref{equation:zeta}, we see that
$$
T_{9}\cap T_{21}=\mathcal{C}_2\cup\zeta\big(\Psi\big)\cup\phi\circ\chi\big(\widetilde{Z}\big).
$$
Moreover, the surfaces $T_{9}$ and $T_{21}$ intersect transversally at a general point of the conic~$\mathcal{C}_2$,
since the surface $\widetilde{\mathcal{S}}$ does not contain the curves $\widetilde{Z}$ and $\widetilde{\Psi}$.
Furthermore, the curve $\zeta(\Psi)$ has degree $10$ by Lemma~\ref{lemma:curves-in-V22-degree-10}(iii).
Thus $\phi\circ\chi(\widetilde{Z})$ is also a curve of degree $10$.
\end{remark}

\begin{remark}
\label{remark:E-Q-u-2}
Suppose that $u=2$.
Let $\widetilde{Z}$ be an irreducible $G$-invariant curve contained in the surface $E_{Q_u}$ that is different from the curve $\widetilde{\Gamma}$.
Let us use the notation from the proof of Lemma~\ref{lemma:E-Q-u-two-curves} and Remark~\ref{remark:E-Q-u-two-curves}.
In the proof of Lemma~\ref{lemma:E-Q-u-two-curves},
we showed that both surfaces~$\widetilde{N}_{5}$ and~$\widetilde{N}_{13}$ contain the curve $\widetilde{Z}$.
On the other hand, we have
$$
N_{5}\cap N_{13}=\Gamma\cup\Delta\cup L_{1}\cup L_{2}.
$$
Moreover, the surfaces $N_5$ and $N_{13}$ are not tangent at a general point of the conic~$\Delta$.
This can be checked, for example, using local equations
of the surfaces $N_5$ and $N_{13}$ at the point~\mbox{$(1:0:2:0:2)$}.
Observe also that the surface $N_{5}$ is smooth at the point~\mbox{$(0:0:0:0:1)$},
and the surface $N_{13}$ is smooth at the point $(1:0:0:0:0)$.
Hence, we deduce that
$$
\widetilde{N}_{5}\cap\widetilde{N}_{13}=\widetilde{Z}\cup\widetilde{\Delta}\cup\widetilde{L}_{1}\cup\widetilde{L}_{2},
$$
where $\widetilde{\Delta}$ is the proper transform of the conic $\Delta$.
Moreover, the surfaces $\widetilde{N}_{5}$ and $\widetilde{N}_{13}$ intersect transversally at a general point of the curve~$\widetilde{Z}$.
Indeed, otherwise the curve~$\Gamma$ would be contained in the one-cycle $N_{5}\cdot N_{13}$ with multiplicity at least $3$,
which is impossible, since $H_{Q_u}\cdot N_{5}\cdot N_{13}=18$, and the one-cycle $N_{5}\cdot N_{13}$
also contains the conic $\Delta$ and the lines  $L_{1}$ and $L_{2}$.
Thus, keeping in mind that the curves $\widetilde{L}_{1}$ and $\widetilde{L}_{2}$ are contracted by $\alpha$,
we conclude that
$$
\alpha\big(\widetilde{N}_{5}\big)\cap\alpha\big(\widetilde{N}_{13}\big)=\alpha(\widetilde{Z})\cup\gamma(\Delta).
$$
On the other hand, the degree of the curve $\gamma(\Delta)$ is $4$, one has $-K_{Y_{u}}^3=16$ and
$$
\alpha\big(\widetilde{N}_{5}\big)\sim\alpha\big(\widetilde{N}_{13}\big)\sim -K_{Y_{u}}.
$$
This implies that $\alpha(\widetilde{Z})$ is a curve of degree $12$,
because $\alpha(\widetilde{N}_{5})$ and $\alpha(\widetilde{N}_{13})$ intersect transversally
at general points of the curves $\alpha(\widetilde{Z})$ and~$\gamma(\Delta)$.
Denote by $\widetilde{C}$ the proper transform of the curve $\widetilde{Z}$ on the threefold $\widetilde{V}_u$.
Then
\begin{multline*}
12=\mathrm{deg}\big(\alpha(\widetilde{Z})\big)=-K_{\widetilde{Q}_u}\cdot\widetilde{Z}=-K_{Y_{u}}\cdot\alpha(\widetilde{Z})=-K_{Y_{u}}\cdot\beta(\widetilde{C})=-K_{\widetilde{V}_u}\cdot\widetilde{C}=\\
=\Big(\phi^*\big(H_{V_u}\big)-E_{V_u}\Big)\cdot\widetilde{C}\leqslant\phi^*\big(H_{V_u}\big)\cdot\widetilde{C}=H_{V_u}\cdot\widetilde{C}=\mathrm{deg}\big(\phi(\widetilde{C})\big).
\end{multline*}
\end{remark}

We conclude our investigation of irreducible $G$-invariant curves in $E_{Q_u}$ by the following result,
which also completes the description of irreducible $G$-invariant curves in $V_u$ of degree~$10$ started in Lemma~\ref{lemma:curves-in-V22-degree-10} and Remark~\ref{remark:E-Q-u-2-3}.

\begin{lemma}
\label{lemma:E-Q-u-curves}
Let $\widetilde{Z}$ be an irreducible $G$-invariant curve contained in the surface $E_{Q_u}$.
Then one of the following two possibilities holds.
\begin{itemize}
\item The curve $\widetilde{Z}$ is the curve $\widetilde{\Gamma}$ from Remark~\ref{remark:flop-Gamma}.
The curve $\phi\circ\chi(\widetilde{Z})$ is the conic $\mathcal{C}_2$.
The degree of the curve $\alpha(\widetilde{Z})$ is at least $12$.

\item The curve $\widetilde{Z}$ is the unique  irreducible $G$-invariant
curve in $E_{Q_u}$ not contained in~$\widetilde{\mathcal{S}}$.
If $u\ne\frac{2}{3}$, then $\mathrm{deg}(\phi\circ\chi(\widetilde{Z}))\geqslant 12$.
If $u=\frac{2}{3}$, then $\mathrm{deg}(\phi\circ\chi(\widetilde{Z}))=10$,
and the curve $\phi\circ\chi(\widetilde{Z})$ is contained
in $T_{9}\cap T_{21}$.
\end{itemize}
\end{lemma}

\begin{proof}
The normal bundle of the smooth rational curve $\Gamma$ in $Q_u$ is isomorphic
to~\mbox{$\mathcal{O}_{\mathbb{P}^1}(p)\oplus\mathcal{O}_{\mathbb{P}^1}(q)$}
for some integers $p$ and $q$ such that $p\geqslant q$ and $p+q=16$.
Thus, the exceptional surface~$E_{Q_u}$ is a Hirzebruch surface $\mathbb{F}_n$ for $n=p-q\geqslant 0$.
Denote by $\textbf{s}$ the section of the natural projection $E_{Q_u}\to\Gamma$ such that $\textbf{s}^2=-n$. Then
$-E_{Q_u}\vert_{E_{Q_u}}\sim\textbf{s}+\kappa\textbf{l}$
for some integer $\kappa$. One has
$$
-16=E_{Q_u}^3=\Big(\textbf{s}+\kappa\textbf{l}\Big)^2=-n+2\kappa,
$$
so that $\kappa=\frac{n-16}{2}$.
This implies that $\widetilde{\mathcal{S}}\vert_{E_{Q_u}}\sim\textbf{s}+\frac{n+8}{2}\textbf{l}$.
On the other hand, it follows from Remark~\ref{remark:flop-Gamma} that
$\widetilde{\mathcal{S}}\vert_{E_{Q_u}}=\widetilde{\Gamma}+\textbf{l}_1+\textbf{l}_2$,
where $\textbf{l}_1$ and $\textbf{l}_2$ are the fibers of the natural projection $E_{Q_u}\to\Gamma$ over the points  $(1:0:0:0:0)$ and~\mbox{$(0:0:0:0:1)$}, respectively.
This gives $\widetilde{\Gamma}\sim\textbf{s}+\frac{n+4}{2}\textbf{l}$,
which implies, in particular, that $\widetilde{\Gamma}\ne\textbf{s}$. Hence, we have
$$
0\leqslant\widetilde{\Gamma}\cdot\textbf{s}=\Big(\textbf{s}+\frac{n+4}{2}\textbf{l}\Big)\cdot\textbf{s}=\frac{4-n}{2},
$$
which implies that $n\leqslant 4$.
Thus, we compute
\begin{equation}
\label{equation:curve-degree-intersection}
\mathrm{deg}\big(\alpha(\widetilde{Z})\big)=-K_{\widetilde{Q}_u}\cdot\widetilde{Z}=\Big(3\pi^*\big(H_{Q_u}\big)-E_{Q_u}\Big)\cdot\widetilde{Z}
=\Big(\textbf{s}+\frac{n+20}{2}\textbf{l}\Big)\cdot\widetilde{Z}.
\end{equation}
In particular, if $\widetilde{Z}=\widetilde{\Gamma}$, then \eqref{equation:curve-degree-intersection} gives
$$
\mathrm{deg}\big(\alpha(\widetilde{Z})\big)=
\Big(\textbf{s}+\frac{n+20}{2}\textbf{l}\Big)\cdot\Big(\textbf{s}+\frac{n+4}{2}\textbf{l}\Big)=12.
$$

Let $\widetilde{C}$ be the proper transform of the curve $\widetilde{Z}$ on the threefold $\widetilde{V}_u$, and let $C=\phi(\widetilde{C})$.
If $\widetilde{Z}\ne\widetilde{\Gamma}$, then
\begin{multline}\label{eq:deg-alpha-Z-vs-deg-C}
\mathrm{deg}\big(\alpha(\widetilde{Z})\big)=-K_{\widetilde{Q}_u}\cdot\widetilde{Z}=-K_{Y_{u}}\cdot\alpha(\widetilde{Z})=-K_{Y_{u}}\cdot\beta(\widetilde{C})=-K_{\widetilde{V}_u}\cdot\widetilde{C}=\\
=\Big(\phi^*\big(H_{V_u}\big)-E_{V_u}\Big)\cdot\widetilde{C}\leqslant\phi^*\big(H_{V_u}\big)\cdot\widetilde{C}=H_{V_u}\cdot\widetilde{C}=\mathrm{deg}(C).
\end{multline}

Now let us use the notation from the proof of Lemma~\ref{lemma:E-Q-u-two-curves} and Remark~\ref{remark:E-Q-u-two-curves}.
To complete the proof, we may assume that $\widetilde{Z}$ is the closure of the $\mathbb{C}^\ast$-orbit of the point~\mbox{$\widetilde{M}_{15}^\mu\cap\textbf{l}$.}
Then~$\widetilde{Z}$ is contained in $\widetilde{M}_{15}^\mu$,
it is a section of the natural projection $E_{Q_u}\to\Gamma$,
and it is not contained in $\widetilde{\mathcal{S}}$.
In particular, we have $\widetilde{Z}\ne\widetilde{\Gamma}$.

By Remarks~\ref{remark:E-Q-u-2-3} and \ref{remark:E-Q-u-2}, we may assume that $u\ne\frac{2}{3}$ and $u\ne 2$.
This implies that~$n=0$, cf. Remark~\ref{remark:u-2-3-u-2} below.
Indeed, suppose that $n>0$. Then $\widetilde{Z}=\textbf{s}$ by Lemma~\ref{lemma:E-Q-u-two-curves}, because the curve $\textbf{s}$ is clearly $G$-invariant.
Then it follows from \eqref{equation:curve-degree-intersection} that
$$
\mathrm{deg}\big(\alpha(\widetilde{Z})\big)=-K_{\widetilde{Q}_u}\cdot\widetilde{Z}=\frac{20-n}{2}<10.
$$
Hence, at least one surface among
$\widetilde{N}_{3}$, $\widetilde{N}_{5}$, $\widetilde{N}_{8}$, $\widetilde{N}_{10}$, $\widetilde{N}_{13}$ and $\widetilde{N}_{15}$ contains the curve~$\widetilde{Z}$.
Since $\iota(\widetilde{N}_{3})=\widetilde{N}_{15}$, $\iota(\widetilde{N}_{5})=\widetilde{N}_{13}$ and
$\iota(\widetilde{N}_{8})=\widetilde{N}_{10}$, this implies that
$\widetilde{Z}$ is contained in at least one of the intersections
$\widetilde{N}_{3}\cap\widetilde{N}_{15}$,
$\widetilde{N}_{5}\cap\widetilde{N}_{13}$,
$\widetilde{N}_{8}\cap\widetilde{N}_{10}$.
On the other hand, it follows from Lemma~\ref{lemma:curves-in-V22-degree-10}(vi) that $N_3$ is tangent to $N_{15}$
at a general point of the curve $\Gamma$ if and only if $u=\frac{2}{3}$.
Since we assumed that $u\ne\frac{2}{3}$, we see that
$$
\widetilde{Z}\not\subset\widetilde{N}_{3}\cap\widetilde{N}_{15}.
$$
Likewise, the surface $N_5$ is tangent to the surface $N_{13}$
at a general point of the curve $\Gamma$ if and only if $u=2$.
We showed this in the proof of Lemma~\ref{lemma:E-Q-u-two-curves}.
Similar computations imply that the surface $N_8$ is not tangent to $N_{10}$ at a general point of the curve $\Gamma$.
Therefore, the curve $\widetilde{Z}$ is contained neither in $\widetilde{N}_{5}\cap\widetilde{N}_{13}$ nor in $\widetilde{N}_{8}\cap\widetilde{N}_{10}$.
The obtained contradiction shows that the case $n>0$ is impossible, so that $n=0$.

Since $n=0$, one has $E_{Q_u}\cong\mathbb{P}^1\times\mathbb{P}^1$.
By \eqref{equation:curve-degree-intersection}, we have
$$
-K_{\widetilde{Q}_u}\cdot\widetilde{Z}=\Big(\textbf{s}+10\textbf{l}\Big)\cdot\widetilde{Z}\geqslant\Big(\textbf{s}+10\textbf{l}\Big)\cdot\textbf{s}=10.
$$
This also shows that $-K_{\widetilde{Q}_u}\cdot\widetilde{Z}=10$ if and only if $\widetilde{Z}\sim\textbf{s}$.
However, this case is impossible.
Indeed, if $\widetilde{Z}\sim\textbf{s}$, then the linear system $|\textbf{s}|$ contains at least two irreducible $G$-invariant curves.
On the other hand, we already know from Lemma~\ref{lemma:E-Q-u-two-curves} that
$\widetilde{Z}$ and $\widetilde{\Gamma}\sim\textbf{s}+2\textbf{l}$ are the only irreducible $G$-invariant curves in the surface $E_{Q_u}$.
Hence, using~\eqref{eq:deg-alpha-Z-vs-deg-C} we conclude that
$\mathrm{deg}(C)\geqslant-K_{\widetilde{Q}_u}\cdot\widetilde{Z}\geqslant 11$.

Using Lemma~\ref{lemma:curves-in-V22-degree-10-12}, we see that $V_u$ does not contain
irreducible $G$-invariant curves of degree $1$, $3$, $5$, $7$, $8$ and $9$.
In particular, the threefold $V_u$ does not contain $G$-invariant lines,
which also follows from \cite[Lemma~4.1(i)]{KuznetsovProkhorov}.

By Remark~\ref{remark:pencil-base-locus-easy},
there exists a unique surface in the pencil generated by $T_{15}$ and $T_{15}^\prime$ that contains $C$.
In fact, we know this surface from Remark~\ref{remark:E-Q-u-two-curves}.
It is the image of the surface~$\widetilde{M}_{15}^\mu$ from Remark~\ref{remark:E-Q-u-two-curves},
where $\mu=-\frac{3u^2+16u-16}{4(u-1)^2}$.
Thus, if $\mathrm{deg}(C)=11$, there should be at least one surface among
$T_9$, $T_{10}$, $T_{11}$, $T_{12}$, $T_{13}$, $T_{14}$, $T_{16}$, $T_{17}$, $T_{18}$, $T_{19}$, $T_{20}$, $T_{21}$ that also contains $C$.
But we proved above that none of the surfaces $\widetilde{N}_{3}$, $\widetilde{N}_{5}$, $\widetilde{N}_{8}$, $\widetilde{N}_{10}$, $\widetilde{N}_{13}$, $\widetilde{N}_{15}$
contains the curve $\widetilde{Z}$, so that the surfaces $T_9$, $T_{11}$, $T_{14}$, $T_{16}$, $T_{19}$ and $T_{21}$ do not contain $C$ either.
Similarly, the surfaces  $T_{12}$, $T_{13}$,  $T_{17}$ and $T_{18}$ do not contain the curve $C$,
because the surfaces $H_x$, $H_y$, $H_z$, $H_t$ and $H_w$ do not contain the curve $\Gamma$.
Thus, to complete the proof, we may assume that either $T_{10}$ or $T_{20}$ contains the curve $C$.
Actually, this assumption implies that both surfaces $T_{10}$ and $T_{20}$ contain the curve $C$, since $\iota(T_{10})=T_{20}$.
Note that this case is indeed possible when $u=-2$ by Remark~\ref{remark:u-minus-2} below.

By Lemma~\ref{lemma:curves-in-V22-degree-10-12}, both surfaces $T_{10}$ and $T_{20}$ contain the curves $\zeta(\Delta)$ and $\zeta(\Upsilon)$,
the degree of the curve $\zeta(\Delta)$ is $4$,
and the degree of the curve $\zeta(\Upsilon)$ is $6$.
Since we already know that $\mathrm{deg}(C)\geqslant 11$,
we see that the $G$-invariant one-cycle $T_{10}\cdot T_{20}$ consists of the curves $\zeta(\Delta)$,
$\zeta(\Upsilon)$, $C$ and a $G$-invariant curve of degree $12-\mathrm{deg}(C)$.
Since $V_u$ does not contain $G$-invariant lines, we see that
$$
T_{10}\cdot T_{20}=\zeta(\Delta)+\zeta(\Upsilon)+C,
$$
so that $\mathrm{deg}(C)=12$. This completes the proof of the lemma.
\end{proof}

\begin{remark}
\label{remark:u-2-3-u-2}
If $u\ne\frac{2}{3}$ and $u\ne 2$, then $E_{Q_u}\cong\mathbb{P}^1\times\mathbb{P}^1$,
so that the normal bundle of the curve $\Gamma$ in the quadric $Q_u$ is isomorphic to $\mathcal{O}_{\mathbb{P}^1}(8)\oplus\mathcal{O}_{\mathbb{P}^1}(8)$.
We showed this in the proof of Lemma~\ref{lemma:E-Q-u-curves}.
Vice versa, if $u=\frac{2}{3}$ or $u=2$, then, arguing as in the proof of Lemma~\ref{lemma:E-Q-u-curves}, one can show that $E_{Q_u}\cong\mathbb{F}_{4}$,
so that the normal bundle of the curve $\Gamma$ is $\mathcal{O}_{\mathbb{P}^1}(6)\oplus\mathcal{O}_{\mathbb{P}^1}(10)$ in this case. However, we will not use this information in the sequel.
\end{remark}

\begin{remark}
\label{remark:u-minus-2}
Denote by $\widetilde{M}_{10}$ and $\widetilde{M}_{20}$ the proper transform of the surfaces $M_{10}$ and $M_{20}$ on the threefold $\widetilde{Q}_u$, respectively.
Recall that both $M_{10}$ and $M_{20}$ has quadratic singularity at the point $(1:1:1:1:1)$.
Substituting $x=1$ and $w=\frac{yt}{u}+\frac{u-1}{u}z^2$ into the polynomial~\mbox{$ug_{10}$}, we obtain the polynomial
$ut^2+ty^5-(2u+1)y^2zt+(u-1)y^4z^2+yz^3$.
The quadratic part of its local expansion at the point $(1:1:1:1:1)$ is
$$
u\bar{t}^2+(3-4u)\bar{y}\bar{t}-(2u+1)\bar{t}\bar{z}+(4u+3)\bar{y}^2+(4u-7)\bar{y}\bar{z}+(u+2)\bar{z}^2,
$$
where $\bar{y}=y-1$, $\bar{z}=z-1$ and $\bar{t}=t-1$.
Similarly, substituting $x=1$ and $w=\frac{yt}{u}+\frac{u-1}{u}z^2$ into the polynomial $u^3g_{20}$, we obtain the polynomial
\begin{multline*}
u^3t^4+t^3y^5-(2u^2+u)t^3y^2z+(3u-3)t^2y^4z^2+(-2u^3+u^2+2u)t^2yz^3+\\
+(3u^2-6u+3)ty^3z^4+(u^2-u)tz^5+(u^3-3u^2+3u-1)y^2z^6.
\end{multline*}
Then the quadratic part of the local expansion of the polynomial $u^2g_{20}$ is
\begin{multline*}
(4u^2-5u+2)\bar{t}^2+(4-4u^2-u)\bar{y}\bar{t}-(12u^2-17u+8)\bar{t}\bar{z}+\\
+(u^2+4u+2)\bar{y}^2+(6u^2-u-8)\bar{y}\bar{z}+(9u^2-14u+8)\bar{z}^2.
\end{multline*}
Both these quadric forms are degenerate,
so that they define reducible conics in~$\mathbb{P}^2_{\bar{y},\bar{z},\bar{t}}$.
\mbox{If~$u\ne-2$,} then these conics do not have common components.
However, if $u=-2$, then the former quadratic form is $(\bar{t}-5\bar{y})(\bar{y}+3\bar{z}-2\bar{t})$,
and the latter quadratic form is $4(\bar{y}-12\bar{z}+7\bar{t})(\bar{y}+3\bar{z}-2\bar{t})$.
Note that the quadratic part of the polynomial \eqref{equation:quadratic-part} is a multiple of $(\bar{y}+3\bar{z}-2\bar{t})^2$.
Thus, if $u=-2$, then $\widetilde{M}_{10}\cap\widetilde{M}_{20}$ contains the irreducible $G$-invariant
curve in $E_{Q_u}$ that is different from the curve $\widetilde{\Gamma}$, see Remark~\ref{remark:flop-Gamma}.
\end{remark}

Recall that $\zeta(\mathcal{S})=\mathcal{C}_2$.
Denote the curves $\zeta(\Delta)$ and $\zeta(\Upsilon)$ by $\mathcal{C}_4$ and $\mathcal{C}_6$, respectively.
Similarly, if $u\ne\frac{2}{3}$, let $\mathcal{C}_{10}=\zeta(\Psi)$ and $\mathcal{C}_{10}^\prime=\zeta(\Psi^\prime)$.
Finally, if $u=\frac{2}{3}$, let $\mathcal{C}_{10}=\zeta(\Psi)$ and let $\mathcal{C}_{10}^\prime=\phi\circ\chi(\widetilde{Z})$,
where $\widetilde{Z}$ is the irreducible $G$-invariant
curve in $E_{Q_u}$ that is different from the curve $\widetilde{\Gamma}$.

\begin{proposition}
\label{proposition:curves}
Let $C$ be an irreducible $G$-invariant curve in $V_u$ with $\mathrm{deg}(C)<12$.
Then either $C=\mathcal{C}_2$, or $C=\mathcal{C}_4$, or $C=\mathcal{C}_6$, or $C=\mathcal{C}_{10}$, or $C=\mathcal{C}_{10}^\prime$.
\end{proposition}

\begin{proof}
We may assume that $C\ne \mathcal{C}_2$.
Denote by $\widetilde{C}$ the proper transform of the curve $C$ on the threefold $\widetilde{V}_{u}$.
By Remark~\ref{remark:flop}, the curve $\widetilde{C}$ is not flopped by $\chi^{-1}$.
Denote by $\widetilde{Z}$ the proper transform of the curve $\widetilde{C}$ on the threefold $\widetilde{Q}_u$.
Then $\widetilde{Z}$ is not contracted by $\pi$, since $Q_u$ does not have $G$-fixed points by Lemma~\ref{lemma:points}.

Let $Z=\pi(\widetilde{Z})$.
Then $Z$ is an irreducible $G$-invariant curve.
Hence, the curve $Z$ is either the curve $\Theta_{\pm}$,
or the curve $\Theta_{a,b}$ for some $(a:b)\in\mathbb{P}^1$.
Therefore, if $Z$ is not contained in $\mathcal{S}$, the required assertion follows from Lemmas~\ref{lemma:curves-in-V22-degree-10-12} and \ref{lemma:curves-in-V22-degree-10}.
Thus, we may assume that $Z\subset\mathcal{S}$, which implies that $Z=\Gamma$, because $C\ne\mathcal{C}_2$ by assumption.
This simply means that $\widetilde{Z}$ is contained in the exceptional surface $E_{Q_u}$.
Then $u=\frac{2}{3}$ and $Z=\mathcal{C}_{10}^\prime$ by Lemma~\ref{lemma:E-Q-u-curves}.
\end{proof}

Using Remark~\ref{remark:C-2-T-i} and Lemmas~\ref{lemma:curves-in-V22-degree-10} and \ref{lemma:E-Q-u-curves}, we see that
\begin{equation}
\label{equation:T9-T21}
T_9\cdot T_{21}=\mathcal{C}_{10}+\mathcal{C}_{10}^\prime+\mathcal{C}_2.
\end{equation}

\section{Anticanonical pencil}
\label{section:pencil}

Let $\mathcal{P}_{Q_u}$ be the pencil of surfaces in $|5H_{Q_u}|$ that are cut out on $Q_u$ by
$$
\mu_0 g_{15}+\mu_1 g_{15}^{\prime}=0,
$$
where $(\mu_0:\mu_1)\in\mathbb{P}^1$.
Here $g_{15}$ is the polynomial of weight $15$ in \eqref{equation:3-5-6-7-8-9-10-11-12-13-15},
and $g_{15}^{\prime}$ is the polynomial
of weight $15$ in~\eqref{equation:10-15-20}.
Then the pencil $\mathcal{P}_{Q_u}$ is free from base components.

Denote by $\mathcal{P}_{V_u}$ the proper transform of the pencil $\mathcal{P}_{Q_u}$ on the threefold $V_u$.
Then $\mathcal{P}_{V_u}$ is generated by the irreducible surfaces $T_{15}$ and $T_{15}^\prime$,
and it contains all $G$-invariant surfaces in the linear system $|-K_{V_u}|$.
This follows from \eqref{equation:zeta}.

By Lemma~\ref{lemma:l-1-l-2-T-i},
the base locus of the pencil $\mathcal{P}_{V_u}$ contains the lines $\ell_1$ and $\ell_2$ from Remark~\ref{remark:flop}.
Similarly, we know from Lemma~\ref{lemma:curves-in-V22-degree-10-12}(i) that the base locus of the pencil $\mathcal{P}_{V_u}$ contains the curve $\zeta(\Theta_{\pm})$.
Thus, using Remark~\ref{remark:pencil-base-locus-easy} and Proposition~\ref{proposition:curves}, we obtain

\begin{corollary}
\label{corollary:pencil-base-locus}
The curve $\zeta(\Theta_{\pm})$ is the only irreducible $G$-invariant curve in $V_u$ which is contained in the base locus of the pencil $\mathcal{P}_{V_u}$.
\end{corollary}

Therefore, for every irreducible $G$-invariant curve in $V_u$ that is different from $\zeta(\Theta_{\pm})$,
there exists a unique surface in the pencil $\mathcal{P}_{V_u}$ that contains this curve.
In particular, the pencil $\mathcal{P}_{V_u}$ contains a unique surface that passes through $\mathcal{C}_4$,
and it contains a unique surface that passes through $\mathcal{C}_6$.
Below we describe both of them.

\begin{lemma}
\label{lemma:T-15-prime}
The curve $\mathcal{C}_6$ is not contained in $T_{15}^\prime$.
On the other hand, the curve $\mathcal{C}_4$ is contained in $T_{15}^\prime$.
Moreover, the surface $T_{15}^\prime$ is singular along the curve $\mathcal{C}_4$.
If $u\ne 2$, then~$T_{15}^\prime$ has a non-isolated ordinary double point at a general point of the curve~$\mathcal{C}_4$.
If~\mbox{$u=2$}, then $T_{15}^\prime$ has a non-isolated ordinary triple point at general point of the curve~$\mathcal{C}_4$.
\end{lemma}

\begin{proof}
Recall from \eqref{equation:10-15-20} that
$$
g_{15}^\prime=(u-1)x^2t^3+(u-1)y^3w^2-(u+4)y^2zt^2+(3u+2)xyztw+(4-4u)yz^3t.
$$
Substituting \eqref{equation:Upsilon} into $g_{15}^\prime$, we see that $\Upsilon$ is not contained in $M_{15}^\prime$,
so that $\mathcal{C}_6$ is not contained in $T_{15}^\prime$.
Similarly, substituting \eqref{equation:Delta} into $g_{15}^\prime$, we see that $\Delta$ is contained in $M_{15}^\prime$,
so that $\mathcal{C}_4$ is contained in $T_{15}^\prime$.

To describe the singularity of the surface $T_{15}^\prime$ at a general point of the curve~$\mathcal{C}_4$,
it is enough to describe the singularity of the surface $M_{15}^\prime$ at a general point of the curve~$\Delta$.
The latter point has the form $(\frac{u-1}{u}\tau^2:0:\tau:0:1)$ with $\tau\in\mathbb{C}^\ast$.
Substituting $w=1$ and~\mbox{$x=z^2+\frac{ty-z^2}{u}$} into $g_{15}^\prime=0$ and
multiplying the resulting equation by $\frac{u^2}{u-1}$, we obtain
\begin{equation}
\label{equation:g-15-prime-C-4}
-u(u-2)tyz^3+u^2y^3+(u-1)^2t^3z^4-u(u+2)t^2y^2z+2(u-1)t^4yz^2+t^5y^2=0.
\end{equation}
Thus, at a general point of the curve $\mathcal{C}_4$, the surface $M_{15}^\prime$ has singularity locally isomorphic
to the product of $\mathbb{C}$ and the germ of the curve singularity given by
$$
-u(u-2)ty+u^2y^3+(u-1)^2t^3-u(u+2)t^2y^2+2(u-1)t^4y+t^5y^2=0.
$$
If $u\ne 2$, the quadratic part  $-u(u-2)ty$ of the left hand side is non-degenerate,
so that~$M_{15}^\prime$ has a non-isolated ordinary double point at $P$.
If $u=2$, the above equation becomes $t^3+4y^3-8t^2y^2+2t^4y+t^5y^2=0$,
which defines an ordinary triple point (also known as curve singularity of type $\mathbf{D}_4$),
and the assertion follows.
\end{proof}

\begin{corollary}
\label{corollary:2-3}
If $u=2$, then $\alpha_G(V_u)\leqslant\frac{2}{3}$.
\end{corollary}

Let $g_{15}^{\prime\prime}=ug_{15}+g_{15}^\prime$. Then
$$
g_{15}^{\prime\prime}=(u-1)x^2t^3+(u-1)y^3w^2-4y^2zt^2+(u+2)xyztw-4(u-1)yz^3t+ux^2zw^2.
$$
Denote by $M_{15}^{\prime\prime}$
the surface in the quadric $Q_u$ that is cut out by $g_{15}^{\prime\prime}=0$.
Let $T_{15}^{\prime\prime}$ be its proper transform on the threefold $V_u$.
Then $T_{15}^{\prime\prime}$ is an irreducible surface in $\mathcal{P}_{V_u}$.

\begin{lemma}
\label{lemma:T-15-prime-prime}
The curve $\mathcal{C}_4$ is not contained in $T_{15}^{\prime\prime}$.
On the other hand, the curve $\mathcal{C}_6$ is contained in $T_{15}^{\prime\prime}$.
Moreover, the surface $T_{15}^{\prime\prime}$ is singular along the curve $\mathcal{C}_6$.
If $u\ne\frac{3}{4}$, then~$T_{15}^{\prime\prime}$ has a non-isolated ordinary double point at a general point of the curve~$\mathcal{C}_6$.
If~\mbox{$u=\frac{3}{4}$}, then~$T_{15}^{\prime\prime}$ has a non-isolated tacnodal singularity at a general point of the curve~$\mathcal{C}_6$.
\end{lemma}

\begin{proof}
Substituting \eqref{equation:Delta} into $g_{15}^{\prime\prime}$,
we see that $\Delta\not\subset M_{15}^{\prime\prime}$,
so that $\mathcal{C}_4\not\subset T_{15}^{\prime\prime}$.
Similarly, substituting \eqref{equation:Upsilon} into $g_{15}^{\prime\prime}$,
we see that $\Upsilon\subset M_{15}^{\prime\prime}$,
so that $\mathcal{C}_6\subset T_{15}^{\prime\prime}$.

To describe the singularity of the surface $T_{15}^{\prime\prime}$ at a general point of the curve~$\mathcal{C}_6$,
it is enough to describe the singularity of the surface $M_{15}^{\prime\prime}$  at a general point of the curve $\Upsilon$.
The latter point has the form $P=(0:(1-u)\tau^2:\tau:1:0)$ with $\tau\in\mathbb{C}^\ast$.

Substituting $t=1$ and $y=z^2+u(wx-z^2)$ into $g_{15}^{\prime\prime}=0$ and
dividing the resulting equation by $(u-1)$, we obtain
$$
x^2+(3u-2)z^3xw-(u-1)^3w^2z^6+3u(u-1)^2z^4xw^3-3uw^2x^2z-3u^2(u-1)z^2x^2w^4+u^3w^5x^3=0.
$$
Thus, at a general point of the curve $\mathcal{C}_6$, the surface $M_{15}^{\prime\prime}$
has singularity locally isomorphic to the product of $\mathbb{C}$ and the germ of the curve singularity given by
$$
x^2+(3u-2)xw-(u-1)^3w^2+3u(u-1)^2xw^3-3uw^2x^2-3u^2(u-1)x^2w^4+u^3w^5x^3=0.
$$
If $u\ne\frac{3}{4}$, the quadratic part  $x^2+(3u-2)xw-(u-1)^3w^2$ of the left hand side is non-degenerate,
so that $M_{15}^{\prime\prime}$ has a non-isolated ordinary double point at $P$.
If $u=\frac{3}{4}$, the above equation becomes $w^2+16wx+64x^2+9w^3x-144w^2x^2+27w^4x^2+27w^5x^3=0$.
So,~introducing new auxiliary coordinates $w=v-8x$, we get
\begin{multline*}
v^2-13824x^4+4032vx^3+110592x^6-360v^2x^2+\\
+9v^3x-55296vx^5+10368v^2x^4-884736x^8+552960vx^7-864v^3x^3+\\
+27v^4x^2-138240v^2x^6+17280v^3x^5-1080v^4x^4+27v^5x^3=0.
\end{multline*}
This equation defines a tacnodal point (also known as curve singularity of type $\mathbf{A}_3$),
and the assertion follows.
\end{proof}

\begin{corollary}
\label{corollary:3-4}
If $u=\frac{3}{4}$, then $\alpha_G(V_u)\leqslant\frac{3}{4}$.
\end{corollary}

\begin{proof}
Suppose that $u=\frac{3}{4}$. Recall that $T_{15}^{\prime\prime}\sim -K_{V_u}$.
Since $T_{15}^{\prime\prime}$
has a tacnodal singularity at a general point of the curve $\mathcal{C}_6$ by Lemma~\ref{lemma:T-15-prime-prime},
the log pair $(V_u,\frac{3}{4}T_{15}^{\prime\prime})$ is not Kawamata log terminal.
Hence~\mbox{$\alpha_G(V_u)\leqslant\frac{3}{4}$}.
\end{proof}

\section{Sarkisov links}
\label{section:link}

Let $\mathcal{C}$ be one of the irreducible $G$-invariant curves $\mathcal{C}_4$ or $\mathcal{C}_6$ in the threefold $V_u$,
let~\mbox{$\sigma\colon\widehat{V}_u\to V_u$} be the blow up of the curve $\mathcal{C}$, and let $E_{\sigma}$ be the exceptional surface of~$\sigma$.
Denote by $\widehat{T}_{i}$,  $\widehat{T}_{15}^{\prime}$, $\widehat{T}_{15}^{\prime\prime}$ the proper transforms on  $\widehat{V}_u$ of the surfaces $T_{i}$, $T_{15}^{\prime}$, $T_{15}^{\prime\prime}$, respectively.

\begin{remark}
\label{remark:G-irreducible-curve}
Suppose that $\mathcal{C}=\mathcal{C}_4$.
Then
$\widehat{T}_{15}^{\prime}\sim\sigma^*(H_{V_u})-m^{\prime}E_{\sigma}$,
where $m^\prime=\mathrm{mult}_{\mathcal{C}}(T_{15}^{\prime})$. By Lemma~\ref{lemma:T-15-prime}, one has
$$
m^\prime=\left\{\aligned
&2\ \ \text{if}\ u\ne 2,\\
&3\ \ \text{if}\ u=2.\\
\endaligned
\right.
$$
Moreover, if $u\ne 2$, then $T_{15}^\prime$ has a non-isolated ordinary double point at a general point of the curve~$\mathcal{C}$.
In this case, one has
$$
\widehat{T}_{15}^{\prime}\big\vert_{E_{\sigma}}=\widehat{\mathcal{C}}+\varkappa\big(\textbf{l}_1+\textbf{l}_2\big),
$$
where $\widehat{\mathcal{C}}$ is a $2$-section of the natural projection  $E_{\sigma}\to\mathcal{C}_4$,
the curves $\textbf{l}_1$ and $\textbf{l}_2$ are the fibers of this projection over two $\mathbb{C}^\ast$-fixed points in $\mathcal{C}_4$, respectively,
and $\varkappa$ is a non-negative integer.
Moreover, it can be seen from \eqref{equation:g-15-prime-C-4} that the curve $\widehat{\mathcal{C}}$ is reducible,
so that it consists of two sections of the projection $E_\sigma\to\mathcal{C}$.
However, the curve $\widehat{\mathcal{C}}$ is $G$-irreducible. This follows from \eqref{equation:involution} and \eqref{equation:g-15-prime-C-4}.
\end{remark}

Let us show that the divisor $-K_{\widehat{V}_u}\sim\sigma^*(H_{V_u})-E_{\sigma}$ is nef.

\begin{lemma}
\label{lemma:blow-up-C4}
Suppose that $\mathcal{C}=\mathcal{C}_4$. Then $\sigma^*(H_{V_u})-E_{\sigma}$ is nef.
\end{lemma}

\begin{proof}
Recall from \eqref{equation:Delta} that the conic $\Delta$ is the scheme-theoretic intersection of the surfaces $H_y$ and $H_t$.
Moreover, it follows from \eqref{equation:Delta-polynomials} that
$\mathcal{C}_4$ is contained in the intersection
\begin{equation}
\label{equation:C4-T-i}
T_{10}\cap T_{11}\cap T_{13}\cap T_{14}\cap T_{15}^\prime\cap T_{16}\cap T_{17}\cap T_{19}\cap T_{20}.
\end{equation}

Recall also that $T_{13}$ is the proper transform on $V_u$ of the surface $H_y$,
and the surface~$T_{17}$ is the proper transform on $V_u$ of the surface $H_t$.
Thus, using Remark~\ref{remark:C-2-T-i} and Lemma~\ref{lemma:l-1-l-2-T-i}, we see that the intersection $T_{13}\cap T_{17}$
consists of the curve $\mathcal{C}_4$, the conic $\mathcal{C}_2$,
the lines~$\ell_1$ and~$\ell_2$ from Remark~\ref{remark:flop}, and the proper transform on $V_u$ of the fibers
of $\pi$ over the points~\mbox{$(1:0:0:0:0)$} and $(0:0:0:0:1)$.

Recall that $T_{11}$ is the proper transform on $V_u$ of the surface $N_5$,
and the surface $T_{19}$ is the proper transform on $V_u$ of the surface $N_{13}$.
Since $N_{5}$ contains $\Gamma$ and is smooth at the point $(1:0:0:0:0)$,
the surface $\widetilde{N}_5$ does not contain the fiber of $\pi$ over this point.
Similarly,  the surface $\widetilde{N}_{13}$ does not contain the fiber of $\pi$ over the point~\mbox{$(0:0:0:0:1)$}.
Hence, using Remark~\ref{remark:C-2-T-i} again, we see that
the only curves contained in the intersection~\mbox{$T_{11}\cap T_{13}\cap T_{17}\cap T_{19}$}
are the conic $\mathcal{C}_2$, the curve $\mathcal{C}_4$, and the lines $\ell_1$ and $\ell_2$.

By Remark~\ref{remark:C-2-T-i}, the surface $T_{15}^\prime$ does not contain the conic $\mathcal{C}_2$.
Similarly, it follows from  Lemma~\ref{lemma:l-1-l-2-T-i} that the intersection $T_{10}\cap T_{20}$ contains neither $\ell_1$ nor $\ell_2$.
Thus, we see that $\mathcal{C}_4$ is the only curve contained in the intersection \eqref{equation:C4-T-i}.

The base locus of the linear system $|\sigma^*(H_{V_u})-E_{\sigma}|$ does not contain any curves
outside the exceptional surface $E_\sigma$.
Moreover, the surfaces $T_{13}$ and $T_{17}$ intersect transversally at a general point of the curve~$\mathcal{C}_4$,
because the surfaces $H_y$ and $H_t$ intersect transversally at every point of the conic $\Delta$.
Hence, the base locus of the linear system $|\sigma^*(H_{V_u})-E_{\sigma}|$ does not contain curves,
with the only possible exception of finitely many fibers
of the projection~\mbox{$E_\sigma\to\mathcal{C}_4$}.
This implies the required assertion.
\end{proof}

\begin{lemma}
\label{lemma:blow-up-C6}
Suppose that $\mathcal{C}=\mathcal{C}_6$. Then $\sigma^*(H_{V_u})-E_{\sigma}$ is nef.
\end{lemma}

\begin{proof}
Recall from \eqref{equation:Upsilon} that
the conic $\Upsilon$ is the scheme-theoretic intersection of the surfaces $H_x$ and $H_w$.
Moreover, it follows from \eqref{equation:Upsilon-polynomials} that
$\mathcal{C}_6$ is contained in the intersection
\begin{equation}
\label{equation:C6-T-i}
T_{10}\cap T_{12}\cap T_{14}\cap T_{15}^{\prime\prime}\cap T_{16}\cap T_{18}\cap T_{20}.
\end{equation}

Recall also that $T_{12}$ is the proper transform on $V_u$ of the surface $H_x$,
and the surface $T_{18}$ is the proper transform on $V_u$ of the surface $H_w$.
Moreover, the surface $H_x$ does not contain the point $(1:0:0:0:0)$,
and the surface $H_w$ does not contain the point~\mbox{$(0:0:0:0:1)$}.
Thus, using Remark~\ref{remark:C-2-T-i} and  Lemma~\ref{lemma:l-1-l-2-T-i}, we see that the
intersection~\mbox{$T_{12}\cap T_{18}$}
consists of the curve $\mathcal{C}_6$, the conic $\mathcal{C}_2$,
and the lines $\ell_1$ and $\ell_2$ from Remark~\ref{remark:flop}.

By Remark~\ref{remark:C-2-T-i}, the surface $T_{15}^{\prime\prime}$ does not contain the conic $\mathcal{C}_2$.
Similarly, it follows from  Lemma~\ref{lemma:l-1-l-2-T-i} that the intersection $T_{10}\cap T_{20}$ contains neither $\ell_1$ nor $\ell_2$.
Thus, the curve $\mathcal{C}_6$ is the only curve contained in the intersection \eqref{equation:C6-T-i}.

The base locus of the linear system $|\sigma^*(H_{V_u})-E_{\sigma}|$ does not contain any curves
outside the exceptional surface $E_\sigma$.
Moreover, the surfaces $T_{13}$ and $T_{18}$ intersect transversally at a general point of the curve~$\mathcal{C}_6$,
because the surfaces $H_x$ and $H_w$ intersect transversally at every point of the conic $\Upsilon$.
Therefore,  the base locus of the linear system $|\sigma^*(H_{V_u})-E_{\sigma}|$ does not contain curves
with the only possible exception of
finitely many fibers of the projection $E_\sigma\to\mathcal{C}_6$.
This implies the required assertion.
\end{proof}

We see that $-K_{\widehat{V}_u}$ is nef.
Since $E_\sigma^3=-\mathrm{deg}(\mathcal{C})+2$ and \mbox{$\sigma^*(H_{V_u})\cdot E^{2}=-\mathrm{deg}(\mathcal{C})$},
we compute
$$
-K_{\widehat{V}_u}^3=\left\{\aligned
&12 \ \text{if}\ \mathcal{C}=\mathcal{C}_4,\\
&8 \ \text{if}\ \mathcal{C}=\mathcal{C}_6.\\
\endaligned
\right.
$$
Therefore, the divisor $-K_{\widehat{V}_u}$ is also big.
Thus, it follows from Basepoint-free Theorem that
the linear system $|-nK_{\widehat{V}_u}|$ is free from base points for~\mbox{$n\gg 0$}, see  \cite[Theorem~3.3]{KollarMori98}.
This linear system gives a crepant
birational morphism $\eta\colon\widehat{V}_u\to Y$,
so that $Y$ is a Fano threefold with at most canonical singularities such that $-K_{Y}^3=-K_{\widehat{V}_u}^3$.
Observe that according to the classification of smooth Fano threefolds with Picard rank~$2$,
the threefold~$\widehat{V}_u$ is not Fano. In other words, $\eta$ is not an isomorphism, and~$Y$ is indeed
singular.

\begin{lemma}
\label{lemma:small-morphism}
Suppose that $\mathcal{C}=\mathcal{C}_4$. Then $\eta$ is small if and only if $u\ne 2$.
\end{lemma}

\begin{proof}
If $u=2$, then $\mathrm{mult}_{\mathcal{C}}(T_{15}^{\prime})=3$ by Lemma~\ref{lemma:T-15-prime}, so that
\begin{multline*}
0\leqslant-K_{\widehat{V}_u}^2\cdot\widehat{T}_{15}^{\prime}=\Big(\sigma^*(H_{V_u})-E_{\sigma}\Big)^2\cdot\Big(\sigma^*(H_{V_u})-3E_{\sigma}\Big)=\\
=22+3\sigma^*(H_{V_u})\cdot E_{\sigma}^2+4\sigma^*(H_{V_u})\cdot E_{\sigma}^2-3E_{\sigma}^3=0,
\end{multline*}
which implies that $\widehat{T}_{15}^{\prime}$ is contracted by $\eta$.

We may assume that $u\ne 2$. Then $\mathrm{mult}_{\mathcal{C}}(T_{15}^{\prime})=2$ by Lemma~\ref{lemma:T-15-prime}.
Let $F$ be an irreducible surface in~$\widehat{V}_u$.
Then $F\sim \sigma^*(nH_{V_u})-mE_{\sigma}$ for some integers $n$ and $m$. We compute
\begin{multline*}
-K_{\widehat{V}_u}^2\cdot F=\Big(\sigma^*(H_{V_u})-E_{\sigma}\Big)^2\cdot\Big(\sigma^*(nH_{V_u})-mE_{\sigma}\Big)=\\
=22n+n\sigma^*(H_{V_u})\cdot E_{\sigma}^2+2m\sigma^*(H_{V_u})\cdot E_{\sigma}^2-mE_{\sigma}^3=18n-6m,
\end{multline*}
so that $F$ is contracted by  $\eta$ if and only if $m=3n$.
In particular, the surface $\widehat{T}_{15}^{\prime}$ is not contracted by  $\eta$.
On the other hand, if $F\ne\widehat{T}_{15}^{\prime}$, then
\begin{multline*}
0\leqslant\Big(\sigma^*(H_{V_u})-E_{\sigma}\Big)\cdot F\cdot\widehat{T}_{15}^{\prime}=\Big(\sigma^*(H_{V_u})-E_{\sigma}\Big)\cdot\Big(\sigma^*(nH_{V_u})-mE_{\sigma}\Big)\cdot\Big(\sigma^*(H_{V_u})-2E_{\sigma}\Big)=\\
=22n+2n\sigma^*(H_{V_u})\cdot E_{\sigma}^2+3m\sigma^*(H_{V_u})\cdot E_{\sigma}^2-2mE_{\sigma}^3=14n-8m,
\end{multline*}
so that $m\ne 3n$, which implies that $F$ is also not contracted by  $\eta$.
\end{proof}

Therefore, if $\mathcal{C}=\mathcal{C}_4$ and $u\ne 2$, then it follows from standard computations as in~\mbox{\cite[\S4.1]{IskovskikhProkhorov}}
or \cite{Takeuchi,ArapCutroneMarshburn,CutroneMarshburn} that there exists a $G$-equivariant commutative diagram
\begin{equation}
\label{equation:Cutrone}
\xymatrix{
&\widehat{V}_u\ar@{-->}[rr]^{\rho}\ar@{->}[ld]_{\sigma}\ar@{->}[rd]_{\eta}&& \widehat{V}_{u^\prime}\ar@{->}[ld]^{\eta^\prime}
\ar@{->}[rd]^{\sigma^\prime}&\\%
V_u&&Y&&V_{u^\prime}}
\end{equation} %
where $\rho$ is the flop in the curves contracted by $\eta$,
and the variety $V_{u^\prime}$ is a smooth Fano threefold of type $V_{22}^\ast$ that corresponds to (some) parameter $u^\prime$,
which is possibly different from $u$.
Here the map $\sigma^\prime$ is a birational morphism  that contracts the proper transform of the surface $\widehat{T}_{15}^{\prime}$
to a unique irreducible $G$-invariant (rational normal) curve $\mathcal{C}_4^\prime$ of degree $4$ in~$V_{u^\prime}$.
The diagram \eqref{equation:Cutrone} is Sarkisov link No.~104 in \cite{CutroneMarshburn}.

\begin{remark}
\label{remark:different-Vu}
It would be interesting to know whether the
threefold $V_{u^\prime}$ in \eqref{equation:Cutrone}
is isomorphic to the threefold $V_{u}$ or not, that is, whether $u=u^\prime$ or not.
\end{remark}

\begin{lemma}
\label{lemma:flopping-curves}
Suppose that $\mathcal{C}=\mathcal{C}_4$ and $u\ne 2$. Then $\eta$ does not contract curves in $E_\sigma$.
\end{lemma}

\begin{proof}
The normal bundle of the curve $\mathcal{C}_4$ in $V_u$ is isomorphic to $\mathcal{O}_{\mathbb{P}^1}(p)\oplus\mathcal{O}_{\mathbb{P}^1}(q)$
for some integers $p$ and $q$ such that $p\geqslant q$ and $p+q=2$.
Thus, the exceptional surface~$E_{\sigma}$ is a Hirzebruch surface $\mathbb{F}_n$ for $n=p-q\geqslant 0$.
Denote by $\textbf{s}$ a section of the natural projection $E_{\sigma}\to\mathcal{C}_4$ such that $\textbf{s}^2=-n$,
and denote by $\textbf{l}$ a fiber of this projection. Then~\mbox{$-E_{\sigma}\vert_{E_{\sigma}}\sim\textbf{s}+\kappa\textbf{l}$}
for some integer $\kappa$. One has
$$
-2=E_{\sigma}^3=\Big(\textbf{s}+\kappa\textbf{l}\Big)^2=-n+2\kappa,
$$
so that $\kappa=\frac{n-2}{2}$.
By Remark~\ref{remark:G-irreducible-curve}, one has
$$
\widehat{T}_{15}^{\prime}\big\vert_{E_{\sigma}}= \widehat{\mathcal{C}}+\varkappa\big(\textbf{l}_1+\textbf{l}_2\big),
$$
where $\widehat{\mathcal{C}}$ is a reducible $G$-irreducible $2$-section of the projection  $E_{\sigma}\to\mathcal{C}_4$,
the curves $\textbf{l}_1$ and $\textbf{l}_2$ are the fibers of this projection over two $\mathbb{C}^\ast$-fixed points in $\mathcal{C}_4$, respectively,
and $\varkappa$ is a non-negative integer.
This gives
$$
\widehat{\mathcal{C}}\sim 2\textbf{s}+(n+2-2\varkappa)\textbf{l}.
$$
Since $\widehat{\mathcal{C}}\ne\textbf{s}$, we have
$0\leqslant\widehat{\mathcal{C}}\cdot\textbf{s}=2-n-2\varkappa$,
which gives $n\leqslant 2$. This implies that the divisor
$$
-K_{\widehat{V}_u}\vert_{E_{\sigma}}\sim\textbf{s}+\frac{n+6}{2}\textbf{l}
$$
is ample, and the assertion follows.
\end{proof}

If $\mathcal{C}=\mathcal{C}_6$, then the morphism $\eta$ is never small, since it contracts the surface $\widehat{T}_{15}^{\prime\prime}$.
Indeed, in this case, we have $\widehat{T}_{15}^{\prime\prime}\sim\sigma^*(H_{V_u})-2E_{\sigma}$
by Lemma~\ref{lemma:T-15-prime-prime}, which implies that
$$
K_{\widehat{V}_u}^2\cdot\widehat{T}_{15}^{\prime\prime}=\Big(\sigma^*(H_{V_u})-E_{\sigma}\Big)^2\cdot\Big(\sigma^*(H_{V_u})-2E_{\sigma}\Big)=22+5\sigma^*(H_{V_u})\cdot E_{\sigma}^2-2E_{\sigma}^3=0.
$$
This is a so-called \emph{bad link} (cf. Sarkisov link No.~93 in \cite{ArapCutroneMarshburn}).

\section{The proof}
\label{section:proof}

In this section, we prove Theorem~\ref{theorem:main}. Let
$$
\varepsilon(u)=\left\{\aligned
&\frac{4}{5}\ \ \text{if}\ u\ne\frac{3}{4}\ \text{and}\ u\ne 2,\\
&\frac{3}{4}\ \ \text{if}\ u=\frac{3}{4},\\
&\frac{2}{3}\ \ \text{if}\ u=2.\\
\endaligned
\right.
$$
By Corollaries~\ref{corollary:4-5}, \ref{corollary:2-3} and \ref{corollary:3-4}, we know that $\alpha_G(V_u)\leqslant\varepsilon(u)$.
Thus, by \eqref{equation:Demailly}, to prove Theorem~\ref{theorem:main},
we have to show that the log pair $(V_u, \frac{\varepsilon(u)}{n}\mathcal{D})$ has log canonical singularities
for every  $G$-invariant linear system $\mathcal{D}\subset|-nK_{V_u}|$ and for every positive integer $n$.
For basic properties of singularities of such log pairs, we refer the reader to \cite[Theorem~4.8]{Kollar1997}.

\begin{remark}
\label{remark:KLT-G-invariant-center}
Let $\mathcal{D}$ be a non-empty $G$-invariant linear subsystem
in $|-nK_{V_u}|$ for
some~\mbox{$n\in\mathbb{Z}_{>0}$}.
Fix a positive rational number $\epsilon$.
Suppose that the log  pair $(V_u,\frac{\epsilon}{n}\mathcal{D})$ is strictly log canonical, i.e., log canonical but not Kawamata log terminal.
Let $Z$ be a center of log canonical singularities of the log pair $(V_u,\frac{\epsilon}{n}\mathcal{D})$ (see \cite[Definition~1.3]{Kawamata97}).
Then~$Z$ is $\mathbb{C}^\ast$-invariant. This follows from the existence of an equivariant strong resolution of singularities (see \cite{ReichsteinYoussin2002,Kollar2007}).
\end{remark}

\begin{remark}
\label{remark:mobile-fixed-convexity}
In the assumptions of Remark~\ref{remark:KLT-G-invariant-center}, let $\mathcal{F}$ be the fixed part of the linear system $\mathcal{D}$,
and let $\mathcal{M}$ be its mobile part, so that
$$
\mathcal{D}=\mathcal{F}+\mathcal{M}.
$$
Since $\mathrm{Pic}(V_u)=\mathbb{Z}[-K_{V_u}]$, one has $\mathcal{F}\sim -n_1K_{V_u}$ and
$\mathcal{M}\sim -n_2K_{V_u}$ for some non-negative integers $n_1$ and $n_2$ such that $n_1+n_2=n$.
Then $Z$ is a center of log canonical singularities of either $(V_u,\frac{\epsilon}{n_1}\mathcal{F})$ or $(V_u,\frac{\epsilon}{n_2}\mathcal{M})$,
see \cite[Remark~2.9]{CheltsovShramovV5} and the proof
of~\mbox{\cite[Lemma~2.10]{CheltsovShramovV5}}.
\end{remark}

\begin{remark}
\label{remark:KLT-divisor}
In the assumptions of Remark~\ref{remark:mobile-fixed-convexity},
there is a $\mathbb{C}^\ast$-invariant divisor $D\in\mathcal{D}$.
Then $Z$ is a center of log canonical singularities of the log pair $(V_u,\frac{\epsilon}{2n}(D+\iota(D)))$.
\end{remark}

Hence, to prove Theorem~\ref{theorem:main},
it is enough to show that the log pair $(V_u, \varepsilon(u) D)$ is log canonical
for every $G$-invariant effective $\mathbb{Q}$-divisor $D$ on the threefold $V_u$ such that
$$
D\sim_{\mathbb{Q}}-K_{V_u}.
$$
Moreover, if necessary, we may assume that $D=\frac{1}{n}S$ for some irreducible surface $S$ in the linear system $|-nK_{V_u}|$.
This follows from

\begin{remark}
\label{remark:convexity}
Let $D$ be a $G$-invariant effective $\mathbb{Q}$-divisor $D$ on the threefold $V_u$ such that~\mbox{$D\sim_{\mathbb{Q}}-K_{V_u}$},
and let $Z$ be an irreducible subvariety in $V_u$ such that $Z$ is a center of log canonical singularities of the log pair $(V_u,\epsilon D)$, where $\epsilon$ is a positive rational number.
Suppose that
$$
D=D_1+D_2
$$
for two non-zero effective $G$-invariant $\mathbb{Q}$-divisors
$D_1\sim_{\mathbb{Q}} -\epsilon_1 K_{V_u}$ and $D_2\sim_{\mathbb{Q}} -\epsilon_2K_{V_u}$.
Here $\epsilon_1$ and $\epsilon_2$ are positive rational numbers such that $\epsilon_1+\epsilon_2=1$.
Then either $Z$ is a center of log canonical singularities of the log pair $(V_u,\frac{\epsilon}{\epsilon_1}D_1)$,
or $Z$ is a center of log canonical singularities of the log pair $(V_u,\frac{\epsilon}{\epsilon_2}D_2)$ (or both).
This is well known and easy to prove. See, for instance, \cite[Remark~2.22]{CheltsovShramovUMN} or \cite[Lemma~2.2]{CheltsovPark}.
\end{remark}

The key point in the proof of Theorem~\ref{theorem:main} is the following

\begin{proposition}
\label{proposition:KLT-RNC}
Let $D$ be a $G$-invariant effective $\mathbb{Q}$-divisor on $V_u$ such that $D\sim_{\mathbb{Q}}-K_{V_u}$.
Suppose that $(V_u,\varepsilon(u) D)$ is not log canonical.
Then $(V_u,\varepsilon(u) D)$ is not log canonical at a general point of one of the curves $\mathcal{C}_2$, $\mathcal{C}_4$ or $\mathcal{C}_6$.
\end{proposition}

\begin{proof}
Let $\epsilon$ be a positive rational number such that $(V_u,\epsilon D)$ is strictly log canonical.
Then $\epsilon<\varepsilon(u)$.
Let $Z$ be a minimal center of log canonical singularities of the log pair~\mbox{$(V_u,\epsilon D)$}.
Since $\mathrm{Pic}(V_u)$ is generated by $-K_{V_u}$ and $\epsilon<1$, the center $Z$ is either a point or a curve.
Recall from Remark~\ref{remark:KLT-G-invariant-center} that $Z$ is $\mathbb{C}^\ast$-invariant.
Observe that $\iota(Z)$ is also a minimal center of log canonical singularities of the log pair $(V_u, \frac{\epsilon}{n}D)$.

Now we will use the so-called \emph{perturbation trick}.
For details, see~\mbox{\cite[Lemma~2.4.10]{CheltsovShramovBOOK}}, and the~proofs of \cite[Theorem~1.10]{Kawamata97} and \cite[Theorem~1]{Kawamata98}.
Observe that there exists a~mobile $G$-invariant linear system $\mathcal{B}$ on the threefold $V_u$,
and there are rational numbers~\mbox{$1\gg \epsilon_{1}\geqslant 0$} and~\mbox{$1\gg\epsilon_{2}\geqslant 0$} such that
$$
\big(\epsilon-\epsilon_1\big)D+\epsilon_{2}\mathcal{B}\sim_{\mathbb{Q}} -\theta K_{V_u},
$$
for some positive rational number $\theta<\varepsilon(u)$, the log pair
\begin{equation}
\label{equation:new-log-pair}
\Big(V_u,\big(\epsilon-\epsilon_1\big)D+\epsilon_{2}\mathcal{B}\Big)
\end{equation}
has strictly log canonical singularities,
and the only centers of log canonical singularities of the log pair \eqref{equation:new-log-pair} are $Z$ and $\iota(Z)$.

Observe that the divisor $-(K_{V_u}+(\epsilon-\epsilon_1)D+\epsilon_{2}\mathcal{B})$ is ample, since $\theta<\varepsilon(u)<1$.
Thus, the locus of log canonical singularities of the pair \eqref{equation:new-log-pair} is connected
by the Koll\'ar--Shokurov connectedness principle \cite[Corollary~5.49]{KollarMori98}.
Since there are no $G$-fixed points on $V_u$ by Lemma~\ref{lemma:points},
the center $Z$ is not a point, so that $Z$ is a curve.

By \cite[Proposition~1.5]{Kawamata97}, either $Z=\iota(Z)$, or the centers $Z$ and $\iota(Z)$ are disjoint.
Using the Koll\'ar--Shokurov connectedness, we see that $Z=\iota(Z)$, so that $Z$ is $G$-invariant.

Since $(\theta-\varepsilon(u))K_{V_u}$ is an ample $\mathbb{Q}$-divisor,
using Kawamata subadjunction theorem \cite[Theorem~1]{Kawamata98}, we see that $Z$ is smooth and
$$
(1-\varepsilon(u))K_{V_u}\Big\vert_{Z}\sim_{\mathbb{Q}}\Big(K_{V_u}+(\epsilon-\epsilon_1)D+\epsilon_{2}\mathcal{B}+(\theta-\varepsilon(u))K_{V_u}\Big)\Big\vert_{Z}\sim_{\mathbb{Q}} K_Z+D_Z
$$
for some ample divisor $D_Z$ on the curve $Z$. In particular, we see that $Z$ is rational and
$$
(\varepsilon(u)-1)\mathrm{deg}\big(Z\big)>-2,
$$
which implies that $\mathrm{deg}(Z)<\frac{2}{1-\varepsilon(u)}\leqslant 10$, so that $\mathrm{deg}(Z)\leqslant 9$.
Thus, by Proposition~\ref{proposition:curves},
the curve $Z$ is one of the curves $\mathcal{C}_2$, $\mathcal{C}_4$ or $\mathcal{C}_6$,
which is exactly what we need.
\end{proof}

In the remaining part of this section, we will show that $(V_u, \varepsilon(u) D)$
is log canonical at general points of the curves  $\mathcal{C}_2$, $\mathcal{C}_4$ or $\mathcal{C}_6$
for every $G$-invariant effective $\mathbb{Q}$-divisor $D$ on the threefold $V_u$ such that $D\sim_{\mathbb{Q}}-K_{V_u}$.
We start with the conic $\mathcal{C}_2$.

\begin{lemma}
\label{lemma:lct-C2}
Let $D$ be a $G$-invariant effective $\mathbb{Q}$-divisor on $V_u$ such that $D\sim_{\mathbb{Q}}-K_{V_u}$.
Then the log pair $(V_u,\frac{4}{5}D)$ is log canonical at a general point of the curve~$\mathcal{C}_2$.
\end{lemma}

\begin{proof}
The normal bundle of the conic $\mathcal{C}_2$ in $V_u$
is either isomorphic to $\mathcal{O}_{\mathbb{P}^1}\oplus\mathcal{O}_{\mathbb{P}^1}$,
or isomorphic to $\mathcal{O}_{\mathbb{P}^1}(-1)\oplus\mathcal{O}_{\mathbb{P}^1}(1)$.
Thus, the exceptional surface $E_{V_u}$ is either $\mathbb{P}^1\times\mathbb{P}^1$ or the Hirzebruch surface $\mathbb{F}_2$.

If $E_{V_u}\cong\mathbb{P}^1\times\mathbb{P}^1$,
we denote by $\textbf{s}$ the section of the natural projection $E_{V_u}\to\mathcal{C}_2$ such that $\textbf{s}^2=0$.
Similarly, if $E_{V_u}\cong\mathbb{F}_2$,
we denote by $\textbf{s}$ the section of the projection $E_{V_u}\to\mathcal{C}_2$ such that $\textbf{s}^2=-2$.
If $E_{V_u}\cong\mathbb{P}^1\times\mathbb{P}^1$, then $-E_{V_u}\vert_{E_{V_u}}\sim\textbf{s}$.
Similarly, if $E_{V_u}\cong\mathbb{F}_2$, then
$$
-E_{V_u}\big\vert_{E_{V_u}}\sim\textbf{s}+\textbf{l},
$$
where $\textbf{l}$ is the fiber of the natural projection $E_{V_u}\to\mathcal{C}_2$.

Denote by $\widetilde{D}$ the proper transform of the divisor $D$ on the threefold $\widetilde{V}_u$.
Then
$$
\widetilde{D}\sim_{\mathbb{Q}}\phi^*\big(H_{V_u}\big)-mE_{V_{u}},
$$
where $m=\mathrm{mult}_{\mathcal{C}_2}(D)$. One the other hand, we know that $\mathcal{R}\sim 2\phi^*(H_{V_u})-5E_{V_{u}}$, so that
$$
\widetilde{D}\sim_{\mathbb{Q}}\frac{1}{2}\mathcal{R}+\Big(\frac{5}{2}-m\Big)E_{V_{u}},
$$
which implies that $m\leqslant\frac{5}{2}$, because $E_{Q_u}$ is the proper transform of the surface $\mathcal{R}$ on the threefold $\widetilde{Q}_u$.

Suppose that the log pair $(V_u,\frac{4}{5}D)$ is not log canonical at a general point of the curve~$\mathcal{C}_2$.
Then~\mbox{$m>\frac{5}{4}$}.
Moreover, the surface $E_{V_u}$ contains a $G$-irreducible curve~$\widetilde{C}$ such that $\phi(\widetilde{C})=\mathcal{C}_2$,
and the log pair
\begin{equation}
\label{equation:log-pull-back-2}
\Bigg(\widetilde{V}_u,\frac{4}{5}\widetilde{D}+\Big(\frac{4m}{5}-1\Big)E_{V_u}\Bigg)
\end{equation}
is not log canonical at a general point of the curve $\widetilde{C}$.
Furthermore, since we know that~\mbox{$m\leqslant\frac{5}{2}$}, the curve $\widetilde{C}$ must be a section  of the natural projection $E_{V_u}\to\mathcal{C}_2$.
This fact is well-known. See for instance \cite[Remark~2.5]{CheltsovPark}.
Thus, the curve $\widetilde{C}$ is irreducible.

Applying \cite[Theorem~5.50]{KollarMori98} to \eqref{equation:log-pull-back-2}, we see that the log pair
$(E_{V_u},\frac{4}{5}\widetilde{D}\vert_{E_{V_u}})$
is also not log canonical at a general point of the curve~$\widetilde{C}$.
This simply means that
$$
\frac{4}{5}\widetilde{D}\big\vert_{E_{V_u}}=\theta\widetilde{C}+\Omega
$$
for some rational number $\theta>1$ and some effective $\mathbb{Q}$-divisor $\Omega$ on the surface $E_{V_u}$.

One has $\widetilde{C}\sim \textbf{s}+\kappa\textbf{l}$ for some non-negative integer $\kappa$.
If $E_{V_u}\cong\mathbb{P}^1\times\mathbb{P}^1$, then
$$
\theta\textbf{s}+\theta\kappa\textbf{l}+\Omega\sim_{\mathbb{Q}}\theta\widetilde{C}+\Omega=\frac{4}{5}\widetilde{D}\big\vert_{E_{V_u}}\sim_{\mathbb{Q}}\frac{4m}{5}\textbf{s}+\frac{8}{5}\textbf{l},
$$
so that either $\kappa=0$ or $\kappa=1$.
Thus, in this case we have
$$
-K_{\widetilde{V}_u}\cdot\widetilde{C}=-K_{\widetilde{V}_u}\big\vert_{E_{V_u}}\cdot\widetilde{C}=\big(\textbf{s}+2\textbf{l}\big)\cdot\big(\textbf{s}+\kappa\textbf{l}\big)=2+\kappa\leqslant 3.
$$
Similarly, if $E_{V_u}\cong\mathbb{F}_2$, then
$$
\theta\textbf{s}+\theta\kappa\textbf{l}+\Omega\sim_{\mathbb{Q}}\theta\widetilde{C}+\Omega=\frac{4}{5}\widetilde{D}\big\vert_{E_{V_u}}\sim_{\mathbb{Q}}\frac{4m}{5}\textbf{s}+\frac{8+4m}{5}\textbf{l},
$$
so that $\kappa\leqslant 3$, which gives
$$
-K_{\widetilde{V}_u}\cdot\widetilde{C}=-K_{\widetilde{V}_u}\big\vert_{E_{V_u}}\cdot\widetilde{C}=\big(\textbf{s}+3\textbf{l}\big)\cdot\big(\textbf{s}+\kappa\textbf{l}\big)=1+\kappa\leqslant 4.
$$

We proved that $-K_{\widetilde{V}_u}\cdot\widetilde{C}\leqslant 4$.
Then the degree of the curve $\beta(\widetilde{C})$ is $-K_{\widetilde{V}_u}\cdot\widetilde{C}\leqslant 4$.
This is impossible by Lemmas~\ref{lemma:curves-in-S4} and \ref{lemma:E-Q-u-curves}.
\end{proof}

Now we deal with the curve $\mathcal{C}_{6}$.

\begin{lemma}
\label{lemma:lct-C6}
Let $D$ be an effective $\mathbb{Q}$-divisor on the threefold $V_u$ such that $D\sim_{\mathbb{Q}}-K_{V_u}$.
Suppose that $\mathrm{Supp}(D)$ does not contain $T_{15}^{\prime\prime}$.
Then the log pair $(V_u,D)$ is log canonical at a general point of the curve~$\mathcal{C}_6$.
\end{lemma}

\begin{proof}
Let us use the notation of \S\ref{section:link} with $\mathcal{C}=\mathcal{C}_6$.
Denote by $\widehat{T}_{15}^{\prime\prime}$ the proper transform of the surface $T_{15}^{\prime\prime}$ on the threefold $\widehat{V}_u$.
Then
$$
\widehat{T}_{15}^{\prime\prime}\sim\sigma^*(H_{V_u})-2E_{\sigma}
$$
by Lemma~\ref{lemma:T-15-prime-prime}.

Denote by $\widehat{D}$ the proper transform on $\widehat{V}_u$ of the divisor~$D$. We also let $m=\mathrm{mult}_{\mathcal{C}_6}(D)$.
Using $E_\sigma^3=-4$ and $\sigma^*(H_{V_u})\cdot E^{2}=-6$, we compute
\begin{multline*}
\Big(\sigma^*(H_{V_u})-E_{\sigma}\Big)\cdot\widehat{D}\cdot\widehat{T}_{15}^{\prime\prime}=\Big(\sigma^*(H_{V_u})-E_{\sigma}\Big)\cdot\Big(\sigma^*(H_{V_u})-mE_{\sigma}\Big)\cdot\Big(\sigma^*(H_{V_u})-2E_{\sigma}\Big)=\\
=22+2\sigma^*(H_{V_u})\cdot E_{\sigma}^2+3m\sigma^*(H_{V_u})\cdot E_{\sigma}^2-2mE_{\sigma}^3=10-10m.
\end{multline*}
On the other hand, the divisor $\sigma^*(H_{V_u})-E_{\sigma}$ is nef by Lemma~\ref{lemma:blow-up-C6}.
Thus, we have $m\leqslant 1$, and the assertion follows.
\end{proof}

\begin{corollary}
\label{corollary:lct-C6}
Let $D$ be an effective $\mathbb{Q}$-divisor on $V_u$ such that~\mbox{$D\sim_{\mathbb{Q}}-K_{V_u}$.}
If~$u=\frac{3}{4}$, then the log pair $(V_u,\frac{3}{4}D)$ is log canonical at a general point of the curve~$\mathcal{C}_6$.
If~$u\neq \frac{3}{4}$, then the log pair $(V_u,D)$ is log canonical at a general point of the curve~$\mathcal{C}_6$.
\end{corollary}

\begin{proof}
If $u=\frac{3}{4}$, then $(V_u,\frac{3}{4}T_{15}^{\prime\prime})$ is log canonical at a general point of~$\mathcal{C}_6$ by Lemma~\ref{lemma:T-15-prime-prime}.
Likewise, if $u\ne\frac{3}{4}$, then the pair $(V_u,T_{15}^{\prime\prime})$ is log canonical at a general point of the curve~$\mathcal{C}_6$.
Thus, by Remark~\ref{remark:convexity}, we may assume that $\mathrm{Supp}(D)$ does not contain the surface $T_{15}^{\prime\prime}$.
Now the assertion follows from Lemma~\ref{lemma:lct-C6}.
\end{proof}

Combining Proposition~\ref{proposition:KLT-RNC}, Lemma~\ref{lemma:lct-C2} and Corollary~\ref{corollary:lct-C6}, we obtain

\begin{corollary}
\label{corollary:final}
Let $D$ be a $G$-invariant effective $\mathbb{Q}$-divisor on $V_u$ such that $D\sim_{\mathbb{Q}}-K_{V_u}$.
Suppose that the log pair $(V_u,\varepsilon(u) D)$ is log canonical at a general point of the curve~$\mathcal{C}_4$.
Then the log pair $(V_u,\varepsilon(u) D)$ is log canonical.
\end{corollary}

Finally, we deal with the curve $\mathcal{C}_{4}$ using Corollary~\ref{corollary:final}.

\begin{lemma}
\label{lemma:lct-C4}
Let $D$ be a $G$-invariant effective $\mathbb{Q}$-divisor on $V_u$ such that $D\sim_{\mathbb{Q}}-K_{V_u}$.
Suppose that $\mathrm{Supp}(D)$ does not contain $T_{15}^{\prime}$.
Then the log pair $(V_u,\frac{5}{6}D)$ is log canonical at a general point of the curve~$\mathcal{C}_4$.
\end{lemma}

\begin{proof}
Let us use the notation of \S\ref{section:link} with $\mathcal{C}=\mathcal{C}_4$.
Then $\sigma^*(H_{V_u})-E_{\sigma}$ is nef by Lemma~\ref{lemma:blow-up-C4}.
Denote by $\widehat{D}$ the proper transform on $\widehat{V}_u$ of the divisor $D$.
We also let $m=\mathrm{mult}_{\mathcal{C}_4}(D)$.
If $u=2$, then $\mathrm{mult}_{\mathcal{C}_4}(T_{15}^\prime)=3$ by Remark~\ref{remark:G-irreducible-curve},
so that
\begin{multline*}
0\leqslant\Big(\sigma^*(H_{V_u})-E_{\sigma}\Big)\cdot\widehat{D}\cdot\widehat{T}_{15}^{\prime}=\Big(\sigma^*(H_{V_u})-E_{\sigma}\Big)\cdot\Big(\sigma^*(H_{V_u})-mE_{\sigma}\Big)\cdot\Big(\sigma^*(H_{V_u})-3E_{\sigma}\Big)=\\
=22+3\sigma^*(H_{V_u})\cdot E_{\sigma}^2+4m\sigma^*(H_{V_u})\cdot E_{\sigma}^2-3mE_{\sigma}^3=10-10m,
\end{multline*}
so that $m\leqslant 1$, which implies that the log pair $(V_u,D)$ is log canonical at a general point of the curve~$\mathcal{C}_4$.

Hence, we may assume that $u\ne 2$, so that $\mathrm{mult}_{\mathcal{C}_4}(T_{15}^\prime)=2$ by Remark~\ref{remark:G-irreducible-curve}.
Then
\begin{multline*}
0\leqslant\Big(\sigma^*(H_{V_u})-E_{\sigma}\Big)\cdot\widehat{D}\cdot\widehat{T}_{15}^{\prime}=\Big(\sigma^*(H_{V_u})-E_{\sigma}\Big)\cdot\Big(\sigma^*(H_{V_u})-mE_{\sigma}\Big)\cdot\Big(\sigma^*(H_{V_u})-2E_{\sigma}\Big)=\\
=22+2\sigma^*(H_{V_u})\cdot E_{\sigma}^2+3m\sigma^*(H_{V_u})\cdot E_{\sigma}^2-2mE_{\sigma}^3=14-8m,
\end{multline*}
which gives $m\leqslant\frac{7}{4}$. Let us show that this implies that $(V_u,\frac{5}{6}D)$ is log canonical at a general point of the curve~$\mathcal{C}_4$.

Let $\epsilon=\frac{5}{6}$. Suppose that $(V_u,\epsilon D)$ is not log canonical at a general point of the curve~$\mathcal{C}_4$.
Then the surface $E_\sigma$ contains a $G$-irreducible curve $\widehat{Z}$ such that $\sigma(\widehat{Z})=\mathcal{C}_4$,
and the log pair
\begin{equation}
\label{equation:log-pull-back}
\Bigg(\widehat{V}_u,\epsilon\widehat{D}+\Big(\epsilon m-1\Big)E_\sigma\Big)
\end{equation}
is not log canonical at a general point of the curve~$\widehat{Z}$.
Moreover, since $\epsilon m=\frac{5m}{6}\leqslant\frac{35}{24}<2$,
the curve $\widehat{Z}$ must be a section  of the natural projection $E_\sigma\to\mathcal{C}_4$.
This is well-known. See for instance \cite[Remark~2.5]{CheltsovPark}.

We see that $\widehat{Z}$ is irreducible.
Thus, the curve $\widehat{Z}$ is not contained in $\widehat{T}_{15}^{\prime}$ by Remark~\ref{remark:G-irreducible-curve}.
Moreover, it follows from Lemma~\ref{lemma:flopping-curves} that the curve $\widehat{Z}$ is not contracted by $\eta$,
so that $\widehat{Z}$ is not flopped by~$\rho$.

Denote by $D^\prime$ the proper transform of the divisor $D$ on the threefold $V_{u^\prime}$,
and denote by $T^\prime$ the proper transform of the exceptional surface $E_{\sigma}$ on the threefold $V_{u^\prime}$.
Then the log pair
\begin{equation}
\label{equation:log-pair-V-u-prime}
\Big(V_{u^\prime}, \epsilon D^\prime+\big(\epsilon m-1\big)T^\prime\Big)
\end{equation}
is not log canonical,
because the log pair \eqref{equation:log-pull-back} is is not log canonical at a general point of the curve~$\widehat{Z}$.

Let us compute the class of the divisor $D^\prime$ in the group $\mathrm{Pic}(V_{u^\prime})$,
and the multiplicity of the divisor $D^\prime$ at a general point of the curve~$\mathcal{C}_{4}^\prime$.
Recall from \eqref{equation:Cutrone} that $\mathcal{C}_4^\prime$ is the unique irreducible $G$-invariant curve of degree $4$ in the threefold $V_{u^\prime}$.
We have
$$
\widehat{D}+\big(m-1\big)E_\sigma\sim_{\mathbb{Q}}-K_{\widehat{V}_{u}}.
$$
This implies that $D^\prime+\big(m-1\big)T^\prime\sim_{\mathbb{Q}}-K_{V_{u^\prime}}$,
where $T^\prime$ is the unique surface in the linear system $|-K_{V_{u^\prime}}|$ that is singular along the curve $\mathcal{C}_4^\prime$.
Thus, we have
$$
D^\prime\sim_{\mathbb{Q}} -\big(2-m\big)K_{V_{u^\prime}}.
$$
Similar arguments applied to the divisor $\frac{1}{2-m}D^\prime$ give
$$
-\frac{1}{2-m}K_{V}\sim_{\mathbb{Q}}\frac{1}{2-m}D\sim_{\mathbb{Q}} -\Bigg(2-\frac{\mathrm{mult}_{\mathcal{C}_{4}^\prime}\big(D^\prime\big)}{2-m}\Bigg)K_{V},
$$
so that $\mathrm{mult}_{\mathcal{C}_{4}^\prime}(D^\prime)=3-2m$.

Observe that $\mathrm{mult}_{\mathcal{C}_{4}^\prime}(T^\prime)=2$. Thus, we have
$$
\mathrm{mult}_{\mathcal{C}_{4}^\prime}\Big(\epsilon D^\prime+\big(\epsilon m-1\big)T^\prime\Big)=3\epsilon-2<1,
$$
so that \eqref{equation:log-pair-V-u-prime} is log canonical at a general point of the curve $\mathcal{C}_{4}^\prime$.
On the other hand, we have
$$
\epsilon D^\prime+\big(\epsilon m-1\big)T^\prime\sim_{\mathbb{Q}}-\big(2\epsilon-1\big)K_{V_{u^\prime}}
$$
and $2\epsilon-1=\frac{2}{3}\leqslant\varepsilon(u)$.
Thus, the log pair \eqref{equation:log-pair-V-u-prime} must be log canonical by Corollary~\ref{corollary:final} applied to $V_{u^\prime}$.
The obtained contradiction completes the proof of the lemma.
\end{proof}

\begin{corollary}
\label{corollary:lct-C4}
Let $D$ be an effective $\mathbb{Q}$-divisor on $V_u$ such that $D\sim_{\mathbb{Q}}-K_{V_u}$.
If~$u=2$, then the log pair $(V_u,\frac{2}{3}D)$ is log canonical at a general point of the curve~$\mathcal{C}_4$.
If~$u\neq 2$, then the log pair $(V_u,\frac{5}{6}D)$ is log canonical at a general point of the curve~$\mathcal{C}_4$.
\end{corollary}

\begin{proof}
If $u=2$, then $(V_u,\frac{2}{3}T_{15}^{\prime})$ is log canonical at a general point of $\mathcal{C}_4$ by Lemma~\ref{lemma:T-15-prime}.
Similarly, if $u\ne 2$, then the pair $(V_u,T_{15}^{\prime})$ is log canonical at a general point of the curve~$\mathcal{C}_4$.
Thus, by Remark~\ref{remark:convexity}, we may assume that $\mathrm{Supp}(D)$ does not contain the surface $T_{15}^{\prime}$.
Now the assertion follows from Lemma~\ref{lemma:lct-C4}.
\end{proof}

Combining Corollaries~\ref{corollary:final} and \ref{corollary:lct-C4},
we obtain the assertion of Theorem~\ref{theorem:main}.
Indeed, let $D$ be an effective $\mathbb{Q}$-divisor on the threefold $V_u$ such that
$D\sim_{\mathbb{Q}}-K_{V_u}$.
As we already mentioned, we have to show that the log  pair $(V_u, \varepsilon(u) D)$ is log canonical.
But the log pair~\mbox{$(V_u, \varepsilon(u) D)$} is log canonical at a general point of the curve $\mathcal{C}_4$ by Corollary~\ref{corollary:lct-C4},
so that it is log canonical everywhere by Corollary~\ref{corollary:final}.

\end{document}